\documentclass[10pt]{amsart}

\usepackage{amssymb,amsthm,amsmath}
\usepackage[numbers,sort&compress]{natbib}
\usepackage{color}
\usepackage{graphicx}
\usepackage{tikz}

\title[ ]{On invariant tori in some reversible systems }

\thanks{2020 {\it Mathematics Subject Classification}. 37J40, 70H08.\\
{\it Key words and phrases}. Invariant tori, quasi-periodic, reversible system, degenerate.\\
This work is supported by the National Natural Science Foundation of China (No. 11771093 and No. 11790272).}

\author{Lu Chen}
\address[Lu Chen]{School of Mathematical Sciences,
Fudan University,
Shanghai 200433,
P. R. China} \email{chenlu@zju.edu.cn}

\theoremstyle{plain}
\newtheorem{thm}{Theorem}[section]
 
 \newtheorem{lem}[thm]{Lemma}
 
 \theoremstyle{definition}
 \newtheorem{defn}[thm]{Definition}
 \theoremstyle{remark}
 \newtheorem{rem}[thm]{Remark}
 
 \numberwithin{equation}{section}
\begin{document}

\begin{abstract}
In the present paper, we consider the following reversible system
\begin{equation*}
\begin{cases}
\dot{x}=\omega_0+f(x,y),\\
\dot{y}=g(x,y),
\end{cases}
\end{equation*}
where $x\in\mathbf{T}^{d}$, $y\backsim0\in \mathbf{R}^{d}$, $\omega_0$ is Diophantine, $f(x,y)=O(y)$, $g(x,y)=O(y^2)$ and $f$, $g$ are reversible with respect to the involution G: $(x,y)\mapsto(-x,y)$, that is, $f(-x,y)=f(x,y)$, $g(-x,y)=-g(x,y)$. We study the accumulation of an analytic invariant torus $\Gamma_0$ of the reversible system with Diophantine frequency $\omega_0$ by other invariant tori. We will prove that if the Birkhoff normal form around $\Gamma_0$ is 0-degenerate,  then $\Gamma_0$ is accumulated by other analytic invariant tori, the Lebesgue measure of the union of these tori being positive and the density of the union of these tori at $\Gamma_0$ being one. We will also prove that if the Birkhoff normal form around $\Gamma_0$ is $j$-degenerate ($1\leq j\leq d-1$) and condition (\ref{fb1-100}) is satisfied, then through $\Gamma_0$ there passes an analytic subvariety of dimension $d+j$ foliated into analytic invariant tori  with frequency vector $\omega_0$. If the Birkhoff normal form around $\Gamma_0$ is $d-1$-degenerate, we will prove a stronger result, that is, a full neighborhood of $\Gamma_0$ is foliated into analytic invariant tori  with frequency vectors proportional to $\omega_0$.
\end{abstract}

\maketitle
\section{ Introduction and main results}\label{sec1}
The ``conventional'' KAM theory deals with perturbations of (partially) integrable dynamical systems (first of all, Hamiltonian and reversible ones), and a typical theorem in this setup states that under appropriate conditions, perturbed systems admit many invariant tori carrying quasi-periodic motions and close to the unperturbed invariant tori. However, in recent years, another aspect of the KAM theory has been actively developed: suppose that we have a Hamiltonian system just possessing a Lagrangian quasi-periodic torus $\Gamma_0$ (or a Lagrangian invariant torus $\Gamma_0$ carrying resonant conditionally periodic motions), what then can be said about other Lagrangian quasi-periodic tori in a neighborhood of $\Gamma_0$?

Eliasson-Fayad-Krikorian \cite{a2} considered the Hamiltonian system
\begin{equation}\label{fc1-1}
H(x, y)=\langle\omega_0,y\rangle+P(x,y),\ \ P=O(y^2),
\end{equation}
where $x\in \mathbf{T}^{d}:=\mathbf{R}^{d}/\mathbf{(2\pi Z)}^{d}$,  $y\backsim0\in \mathbf{R}^{d}$, $P\in C^2$. The Hamiltonian system associated to $H$ is given by
\begin{equation}\label{fc1-2}
\begin{cases}
\dot{x}=\frac{\partial H}{\partial y}=\omega_0+O(y),\\
\dot{y}=-\frac{\partial H}{\partial x}=O(y^2).
\end{cases}
\end{equation}
Clearly the torus $\Gamma_0:=\mathbf{T}^{d}\times\left\{y=0 \right\}$ is invariant under the Hamiltonian flow and the induced dynamics is the translation
$$(t,x)\mapsto x_0+t\omega_0.$$
Moreover, this torus is Lagrangian with respect to the canonical symplectic form $dx\wedge dy$ on $\mathbf{T}^{d}\times\mathbf{R}^{d}$. The main results in \cite{a2} are the following:

Let $H$ be real analytic and let the frequency vector $\omega_0$ be Diophantine (so that there exists the Birkhoff normal form of the system around $\Gamma_0$). Then\\
(i) If the BNF of the system around $\Gamma_0$ is 0-degenerate, then the torus $\Gamma_0$ is KAM stable, i.e., $\Gamma_0$ is accumulated by other analytic Lagrangian quasi-periodic tori, the Lebesgue measure of the union of these tori being positive and the density of the union of these tori at $\Gamma_0$ being one.\\
(ii)  If the BNF of the system around $\Gamma_0$ is $j$-degenerate ($1\leq j\leq d-1$), then through
$\Gamma_0$ there passes an analytic coisotropic subvariety of dimension $d+j$ foliated into analytic Lagrangian quasi-periodic tori with frequency vector $\omega_0$.\\
(iii)   If the BNF of the system around $\Gamma_0$ is $d-1$-degenerate, then a full neighborhood of $\Gamma_0$ is foliated into analytic Lagrangian quasi-periodic tori with frequency vectors proportional to $\omega_0$, in particular, the torus $\Gamma_0$ is KAM stable (also see R\"ussmann \cite{b1}).

The results in \cite{a2} also provided a partial answer to conjectured by M.Herman in his ICM-98's-lecture \cite{a1}(i.e., {\it assuming that $\omega_0$ is Diophantine and $H$ is real analytic, in any sufficiently small neighborhood of the origin there exists a set of positive Lebesgue measure of Lagrangian quasi-periodic tori}).

Bounemoura \cite{b2} considered the following Hamiltonian system
\begin{equation}\label{fd1-1}
H(x, y)=\omega_0\cdot y+A(x)y\cdot y+R(x,y),\ \ (x,y)\in \mathbf{T}^{d}\times \mathbf{R}^{d},
\end{equation}
where $\cdot$ denotes the Euclidean inner product, $\omega_0\in \mathbf{R}^{d}$ is a non-resonant vector ($k\cdot \omega_0\neq 0$ for any $k\in\mathbf{Z}^{d}/ \left\{0 \right\}$), $A(x)$ is, for each $x\in \mathbf{T}^{d}$, a square symmetric matrix of size $d$ with real coefficients and $R(x,y)=O(y^3)$ is of order at least 3 in $y$. Let $H$ be $C^{\ell}$-smooth with $\ell$ sufficiently large, and let $\Gamma_0$ be Kolmogorov nondegenerate. Then the torus $\Gamma_0$ is KAM stable (the definition of KAM stability is the same as above, ``analytic'' is replaced with ``smooth'').

Besides the Hamiltonian structure (or symplectic structure for mappings), there is so-called reversible structure for differential equations or mappings on which KAM theory can be constructed. In 1965, Moser \cite{b3} initiated the study of  reversible KAM theory. In 1973, Moser \cite{a19} constructed a KAM theorem for the reversible equation
$$\dot{x}=\omega_0+y+f(x,y), \ \ \ \ \dot{y}=g(x,y),$$
where $f$ and $g$ are analytic in their arguments and reversible with respect to the involution G: $(x,y)\mapsto(-x,y)$, that is,
$$f(-x,y)=f(x,y), \ \ \ g(-x,y)=-g(x,y).$$
The KAM theory for reversible equations (vector-fields) of more general form was deeply
investigated in \cite{a4, a5, a6, a11, a12, a13, a14, a17, a18}. In addition, there are many applications of reversible KAM theory (see \cite{a7, a8, a9, a10, a16}).

In the present paper,  we will study the following reversible system
\begin{equation}\label{fc1-3}
\begin{cases}
\dot{x}=\omega_0+f(x,y),\\
\dot{y}=g(x,y),
\end{cases}
\end{equation}
where $x\in\mathbf{T}^{d}$, $y\backsim0\in \mathbf{R}^{d}$, $\omega_0$ is Diophantine, $f(x,y)=O(y)$, $g(x,y)=O(y^2)$ and $f$, $g$ are reversible with respect to the involution G: $(x,y)\mapsto(-x,y)$, that is, $f(-x,y)=f(x,y)$, $g(-x,y)=-g(x,y)$.

A vector $\omega_0$ is said to be Diophantine if there exist constants $0<\gamma_0<1$ and $\tau_0>d-1$ such that
\begin{equation}\label{fd1-2}
| \langle k,\omega_0\rangle|\geq\frac{\gamma_0}{| k|^{\tau_0}},\ \ \ \forall k\in\mathbf{Z}^{d}/ \left\{0 \right\}.
\end{equation}
We then use the notation $\omega_0\in DC(\gamma_0,\tau_0)$ where $\tau_0$ is the Diophantine exponent of $\omega_0$
and $\gamma_0$ is its Diophantine constant.

Clearly, the torus $\Gamma_0:=\mathbf{T}^{d}\times\left\{y=0 \right\}$ of (\ref{fc1-3}) is an analytic invariant torus with frequency $\omega_0$, i.e., $\Gamma_0$ is an analytic $d$-torus that is invariant under both the system (\ref{fc1-3}) and the reversing involution G and carries quasi-periodic motions with frequency $\omega_0$.

Let
  \begin{equation*}
\begin{cases}
\dot{\phi}=f_F(\mu),\\
\dot{\mu}=0
\end{cases}
\end{equation*}
be the BNF of (\ref{fc1-3}) around $\Gamma_0$, that is, a uniquely defined formal power series in the $\mu$ variable as soon as $\omega_0$ is Diophantine (see Lemma \ref{lem2-3} and Remark \ref{rem2-4}). We say that the BNF of (\ref{fc1-3}) around $\Gamma_0$ is $j$-degenerate if there exist $j$ orthonormal vectors $\gamma_1,\cdots,\gamma_j$  such that, for every $\mu\backsim0\in \mathbf{R}^{d}$,
\begin{equation*}
\langle f_{F}(\mu), \gamma_i\rangle=0,\ \ \forall1\leq i\leq j,
\end{equation*}
but no $j+1$ orthonormal vectors with this property. Since $\omega_0\neq0$ clearly $j\leq d-1$. 0-degenerate is also said to be nondegenerate.

In the present article, we will prove the following theorem:

 \begin{thm}\label{thm1-1} Let $B(r):=\left\{y\in \mathbf{R}^{d}: |y|=\sqrt{y_1^2+\cdots+y_d^2}<r \right\}$ be a d-dimensional ball in $\mathbf{R}^{d}$, centered at the origin, with a small radius $r>0$. Assume $f(x,y)=O(y)$, $g(x,y)=O(y^2)${\rm:} $\mathbf{T}^{d}\times B(r)\rightarrow\mathbf{R}^d$ are analytic, where $f(x,y)=O(y)$, $g(x,y)=O(y^2)$ mean that $f=\sum_{|\alpha|\geq1, \alpha\in\mathbf{N}^d}f_{\alpha}(x)y^{\alpha}$ and $g=\sum_{|\beta|\geq2, \beta\in\mathbf{N}^d}g_{\beta}(x)y^{\beta}$ are power series in $y\in B(r)$ which are convergent uniformly for $x\in\mathbf{T}^{d}$. In addition, assume that $\omega_0\in DC(\gamma_0,\tau_0)$.
 Then:\\
 {\rm(i)} If the BNF of {\rm (\ref{fc1-3})} around $\Gamma_0$ is 0-degenerate, then $\Gamma_0$ is accumulated by other analytic invariant tori of {\rm(\ref{fc1-3})}, the Lebesgue measure of the union of these tori being positive and the density of the union of these tori at $\Gamma_0$ being one.\\
{\rm(ii)} If the BNF of {\rm (\ref{fc1-3})} around $\Gamma_0$ is $j$-degenerate {\rm(}$1\leq j\leq d-1${\rm)} and for every $\mu\backsim0\in \mathbf{R}^{d}$,
\begin{equation}\label{fb1-100}
\frac{\partial (f_{F})_i(\mu)}{\partial \mu_{l}}\equiv\frac{\partial (f_{F})_l(\mu)}{\partial \mu_i}, \ \ \ \  \forall1\leq i,l\leq d,
\end{equation}
where $f_{F}=((f_F)_1,\cdots, (f_F)_d )^{T}$, $\mu=(\mu_1,\cdots,\mu_d)^{T}$.
  Then through $\Gamma_0$ there passes an analytic subvariety of dimension $d+j$ foliated into analytic
invariant tori of {\rm (\ref{fc1-3})} with frequency vector $\omega_0$; \\
{\rm(iii)} If the BNF of {\rm (\ref{fc1-3})} around $\Gamma_0$ is $d-1$-degenerate, we have a stronger result than (ii), that is,  a full neighborhood of $\Gamma_0$ is foliated into analytic invariant tori of {\rm(\ref{fc1-3})} with frequency vectors proportional to $\omega_0$.
\end{thm}

Let us introduce the plan of proof. In Section \ref{sec2}, we perform $N$ times Birkhoff normal form changes such that (\ref{fc1-3}) can be written as a new form (see Lemma \ref{lem2-1} and (\ref{fc2-2})) and give the BNF of (\ref{fc1-3}) (see Lemma \ref{lem2-3} and Remark \ref{rem2-4}). In Section \ref{sec3}, we introduce the flat function and flat operator, which will be used in Section \ref{sec5}. In Sections \ref{sec4}-\ref{sec6},  we formulate a KAM counterterm theorem, specifically, in Section \ref{sec4}, we give a iterative lemma, in Section \ref{sec5}, we give the proof of the iterative lemma, in Section \ref{sec6}, we derive the KAM counterterm theorem by the iterative lemma.  In Section \ref{sec7}, we derive Theorem \ref{thm1-1} by the KAM counterterm theorem.

\section{Some notations and the BNF}\label{sec2}
Let us here introduce some notations. Let $B(c)=\left\{y\in \mathbf{R}^{d}: \sqrt{y_1^2+\cdots+y_d^2}<c \right\}$ with any $c>0$. By $B_{\mathbf{C}}(c)=\left\{z\in \mathbf{C}^{d}: \sqrt{|z_1|^2+\cdots+|z_d|^2}<c \right\}$ denote the complexization of $B(c)$. Let $\mathbf{T}_s^d=\left\{x\in \mathbf{C}^{d}/\mathbf{(2\pi Z)}^{d}: |{\rm Im}x_j|< s,j=1,\cdots,d \right\}$ be the complexization of $\mathbf{T}^d$ with width $s>0$. For $s>0$, $c>0$, $c_1>0$, $c_2>0$, let
$$D(s,c,c_1,c_2)=\mathbf{T}_{s}^d\times B_{\mathbf{C}}(c)\times B(c_1)\times B(c_2).$$
For $(x,y,\xi,\eta)\in D(s,c,c_1,c_2)$, we call that $x$, $y$, $\xi$, $\eta$ are angle variable, action variable, parameter, parameter, respectively.

We define
$$C^{\omega,\infty}(\mathbf{T}_{s}^d\times B_{\mathbf{C}}(c)\times B(c_1)\times B(c_2))=C^{\omega,\infty}(D(s,c,c_1,c_2))$$
be the set of $C^{\infty}$ function (possibly vector valued or matrix valued)
$$f:D(s,c,c_1,c_2)\ni (x,y,\xi,\eta)\mapsto f(x,y,\xi,\eta)\in \mathbf{C} \ or\ \mathbf{C}^{d}\ or \ \mathbf{C}^{d\times d}$$
such that for all $(\xi,\eta)\in B(c_1)\times B(c_2)$,
$$f_{\xi,\eta}:\mathbf{T}_{s}^d\times B_{\mathbf{C}}(c) \ni (x,y)\mapsto f(x,y,\xi,\eta)$$
is real analytic, and that for all $(x,y)\in \mathbf{T}_{s}^d\times B_{\mathbf{C}}(c)$,
$$f_{x,y}:B(c_1)\times B(c_2)\ni (\xi,\eta)\mapsto f(x,y,\xi,\eta)$$
is $C^{\infty}$ function. Fix an integer $\ell\geq0$. Define for $\ell\geq0$,
$$\|f\|_{D(s,c,c_1,c_2),\ell}=\sup_{|\alpha|\leq\ell,\alpha\in\mathbf{N}^d,
(x,y,\xi,\eta)\in D(s,c,c_1,c_2)}|\partial^{\alpha}_{\eta}f(x,y,\xi,\eta)|.$$
When $f\in C^{\omega,\infty}(D(s,c,c_1,c_2))$ is independent of  $\eta$, we write $f=f(x,y,\xi):\mathbf{T}_{s}^d\times B_{\mathbf{C}}(c)\times B(c_1)\rightarrow \mathbf{C}$ or $\mathbf{C}^{d}$ or $\mathbf{C}^{d\times d}$, and call $f\in C^{\omega,\infty}(D(s,c,c_1,0))$ and define
$$\|f\|_{D(s,c,c_1,0),\ell}=\|f\|_{D(s,c,c_1,0),0}=\sup_{(x,y,\xi)\in \mathbf{T}_{s}^d\times B_{\mathbf{C}}(c)\times B(c_1)}|f(x,y,\xi)|.$$
When $f\in C^{\omega,\infty}(D(s,c,c_1,c_2))$ is independent of one of variables $x$, $y$, $\xi$, say, $x$, we write $f=f(y,\xi,\eta):B_{\mathbf{C}}(c)\times B(c_1)\times B(c_2)\rightarrow \mathbf{C}$ or $\mathbf{C}^d$ or $\mathbf{C}^{d\times d}$, and call $f\in C^{\omega,\infty}(D(0,c,c_1,c_2))$ and define
$$\|f\|_{D(0,c,c_1,c_2),\ell}=\sup_{|\alpha|\leq\ell,\alpha\in\mathbf{N}^d,(y,\xi,\eta)\in B_{\mathbf{C}}(c)\times B(c_1)\times B(c_2)}|\partial^{\alpha}_{\eta}f(y,\xi,\eta)|.$$
Similarly, we define $f=f(\xi,\eta)\in C^{\omega,\infty}(D(0,0,c_1,c_2))$, $\|f\|_{D(0,0,c_1,c_2),\ell}$, $f=f(\xi)\in C^{\omega,\infty}(D(0,0,c_1,0))$, $\|f\|_{D(0,0,c_1,0),\ell}$, and so on. Here we do not give out the details. Clearly,
$$C^{\omega,\infty}(D(0,c,c_1,c_2)), C^{\omega,\infty}(D(0,0,c_1,c_2)), and\ so\ on \subset C^{\omega,\infty}(D(s,c,c_1,c_2)).$$
In addition, we call $f=O^j(\xi)$ if  and only if
$$\partial^{\alpha}_{\xi}f(x,y,\xi,\eta)|_{\xi=0}=0, \ \ \forall|\alpha|<j.$$
Similarly, we call $f=O^j(y)$ if  and only if
$$\partial^{\alpha}_{y}f(x,y,\xi,\eta)|_{y=0}=0, \ \ \forall|\alpha|<j.$$
For a formal power series
$$f=f(x,y,\xi,\eta)=\sum_{\alpha,\bar{\alpha}\in\mathbf{N}^d}f_{\alpha,\bar{\alpha}}(x,\eta)y^{\alpha}\xi^{\bar{\alpha}}$$
whose coefficients $f_{\alpha,\bar{\alpha}}(x,\eta)\in C^{\omega,\infty}(D(s,0,0,c_2))$. We denote by
$$[f(x,y,\xi,\eta)]_j=\sum_{|\alpha|+|\bar{\alpha}|=j}f_{\alpha,\bar{\alpha}}(x,\eta)y^{\alpha}\xi^{\bar{\alpha}}$$
the homogenous component of degree $j$.
 For $f=f(x,y,\xi,\eta)$, $(x,y,\xi,\eta)\in D(s,c,c_1,c_2)$, we denote by
$$\hat{f}(k,y,\xi,\eta)=\frac{1}{(2\pi)^{d}}\int_0^{2\pi}\cdots\int_0^{2\pi}f(x,y,\xi,\eta)e^{-i\langle k,x\rangle}dx,\ k\in\mathbf{Z}^{d}$$
the $k$-Fourier coefficient of $f$. When $f$ takes up a lot of space, we by $f^{\hat{}}(k,y,\xi,\eta)$  instead of $\hat{f}(k,y,\xi,\eta)$ denote the $k$-Fourier coefficient of $f$. Similarly, for $f=f(x)$, we can define $\hat{f}(k)$, for $f=f(x,y)$, we can define $\hat{f}(k,y)$ and so on.

We write $f=O(y^n)$ ($n=0,1,2,\cdots$), if $f\in C^{\omega,\infty}(D(s,c,c_1,c_2))$ and
$$f=f(x,y,\xi,\eta)=\sum_{\alpha\in\mathbf{N}^d,|\alpha|\geq n}f_{\alpha}(x,\xi,\eta)y^{\alpha}$$
for $|y|<c$ and uniformly for $(x,\xi,\eta)\in \mathbf{T}^d_{s}\times B(c_1)\times B(c_2)$ with some $s>0$, $c_1>0$, $c_2>0$. Here we do not care about the size of $s>0$, $c_1>0$, $c_2>0$.

In the whole of the present paper we always denote by C a universal constant which may be different in different places.

Recall that $f(x,y)=O(y)$, $g(x,y)=O(y^2)$ in (\ref{fc1-3}) are analytic in $\mathbf{T}^{d}\times B(r)$. So the definition domain of the functions $f(x,y)$, $g(x,y)$ can be extended to the complexization domain $\mathbf{T}^{d}_s\times B_{\mathbf{C}}(r)$. The extended functions $f(x,y)$, $g(x,y)$ are called to be real analytic, that is, $f(x,y)$, $g(x,y)$ are analytic in the complexization domain $\mathbf{T}^{d}_s\times B_{\mathbf{C}}(r)$ and are real when the argument vector $(x,y)$ are real.

Now, let  $0<s<1$ fixed and $0<r\ll s<1$. By replacing $(x,y)$ with $(\tilde{x},\tilde{y})$, we can rewrite (\ref{fc1-3}) as
\begin{equation}\label{fc2-1}
\begin{cases}
\dot{\tilde{x}}=\omega_0+\sum_{|\alpha|\geq1}f_{\alpha}^{(0)}(\tilde{x})\tilde{y}^{\alpha},\\
\dot{\tilde{y}}=\sum_{|\beta|\geq2}g_{\beta}^{(0)}(\tilde{x})\tilde{y}^{\beta},
\end{cases}
\end{equation}
where $(\tilde{x},\tilde{y})\in D(s,r,0,0)$, $f_{\alpha}^{(0)}, g_{\beta}^{(0)}\in C^{\omega,\infty}(D(s,0,0,0))$, $f_{\alpha}^{(0)}(-\tilde{x})=f_{\alpha}^{(0)}(\tilde{x})$, $g_{\beta}^{(0)}(-\tilde{x})=-g_{\beta}^{(0)}(\tilde{x})$ and  $$\|\sum_{|\alpha|\geq1}f_{\alpha}^{(0)}(\tilde{x})\tilde{y}^{\alpha}\|_{D(s,r,0,0),0}\leq Cr, \ \|\sum_{|\beta|\geq2}g_{\beta}^{(0)}(\tilde{x})\tilde{y}^{\beta}\|_{D(s,r,0,0),0}\leq Cr^2.$$
 Let $N>10d$, $s'_0=s,\ s'_j=s(1-\frac{j}{2N}),\ r'_0=r,\ r'_j=r(1-\frac{j}{2N}),\ j=1,2,\cdots,N.$
 Then we will perform $N$ times Birkhoff normal form changes such that (\ref{fc2-1}) can be written as a new form (\ref{fc2-2}). First, we have the following lemma.
\begin{lem}\label{lem2-1}  Suppose that we have had $m+1$ {\rm(}$m=0,1,2,\cdots,N-1${\rm)} diffeomorphisms $\tilde{\Phi}_0=\tilde{\Psi}_0^{-1}=id$, $\tilde{\Phi}_1=\tilde{\Psi}_1^{-1}$, $\cdots$, $\tilde{\Phi}_m=\tilde{\Psi}_m^{-1}$ with
$$\tilde{\Psi}_j:D(s'_j,r'_j,0,0)\rightarrow D(s'_{j-1},r'_{j-1},0,0), \ j=1,2,\cdots,m$$
of the form
$$\tilde{\Phi}_j:x=\tilde{x}+\sum_{|\alpha|=j}u_{\alpha}(\tilde{x})\tilde{y}^{\alpha}, \ y=\tilde{y}+\sum_{|\beta|=j+1}v_{\beta}(\tilde{x})\tilde{y}^{\beta},\ j=1,2,\cdots,m,$$
where $u_{\alpha}, v_{\beta}\in C^{\omega,\infty}(D(s'_j,0,0,0))$, and
$$\|\sum_{|\alpha|=j}u_{\alpha}(\tilde{x})\tilde{y}^{\alpha}\|_{D(s'_j,r'_j,0,0),0}\leq Cr^{j},\ \|\sum_{|\beta|=j+1}v_{\beta}(\tilde{x})\tilde{y}^{\beta}\|_{D(s'_j,r'_j,0,0),0}\leq Cr^{j+1}, \ j=1,2,\cdots,m$$
such that system {\rm(\ref{fc2-1})} is changed by $\tilde{\Phi}^{(m)}=\tilde{\Phi}_0\circ\cdots\tilde{\Phi}_m$ into
 \begin{equation}\label{fc2-3}
\begin{cases}
\dot{x}=\omega_0+\sum_{0\leq|\alpha|\leq m}d_{\alpha}y^{\alpha}+\sum_{|\alpha|\geq m+1}f_{\alpha}^{(m)}(x)y^{\alpha},\\
\dot{y}=\sum_{|\beta|\geq m+2}g_{\beta}^{(m)}(x)y^{\beta},
\end{cases}
\end{equation}
where $d_{\alpha}$ are constants, $d_0=0$, $f_{\alpha}^{(m)}, g_{\beta}^{(m)}\in C^{\omega,\infty}(D(s'_m,0,0,0))$, $f_{\alpha}^{(m)}(-x)=f_{\alpha}^{(m)}(x)$, $g_{\beta}^{(m)}(-x)=-g_{\beta}^{(m)}(x)$ and  $$\|\sum_{|\alpha|\geq m+1}f_{\alpha}^{(m)}(x)y^{\alpha}\|_{D(s'_m,r'_m,0,0),0}\leq Cr^{m+1},$$  $$\|\sum_{|\beta|\geq m+2}g_{\beta}^{(m)}(x)y^{\beta}\|_{D(s'_m,r'_m,0,0),0}\leq Cr^{m+2}.$$
Then there is a diffeomorphism $\tilde{\Phi}_{m+1}=\tilde{\Psi}_{m+1}^{-1}$ with
$$\tilde{\Psi}_{m+1}:D(s'_{m+1},r'_{m+1},0,0)\rightarrow D(s'_m,r'_m,0,0)$$
of the form
$$\tilde{\Phi}_{m+1}:\theta=x+\sum_{|\alpha|=m+1}u_{\alpha}(x)y^{\alpha},\
\rho=y+\sum_{|\beta|=m+2}v_{\beta}(x)y^{\beta},$$
where $u_{\alpha}, v_{\beta}\in C^{\omega,\infty}(D(s'_{m+1},0,0,0)),$ \ \ \ $\|\sum_{|\alpha|=m+1}u_{\alpha}(x)y^{\alpha}\|_{D(s'_{m+1},r'_{m+1},0,0),0}\leq Cr^{m+1}$,  \ \ \ $\|\sum_{|\beta|=m+2}v_{\beta}(x)y^{\beta}\|_{D(s'_{m+1},r'_{m+1},0,0),0}\leq Cr^{m+2}$ such that system {\rm(\ref{fc2-3})} is changed by $\tilde{\Phi}_{m+1}$ into
 \begin{equation}\label{fcb2-1}
\begin{cases}
\dot{\theta}=\omega_0+\sum_{0\leq|\alpha|\leq m+1}d_{\alpha}\rho^{\alpha}+\sum_{|\alpha|\geq m+2}f_{\alpha}^{(m+1)}(\theta)\rho^{\alpha},\\
\dot{\rho}=\sum_{|\beta|\geq m+3}g_{\beta}^{(m+1)}(\theta)\rho^{\beta},
\end{cases}
\end{equation}
where $d_{\alpha}$ are constants, $d_0=0$, $f_{\alpha}^{(m+1)}(\theta), g_{\beta}^{(m+1)}(\theta)\in C^{\omega,\infty}(D(s'_{m+1},0,0,0))$, $f_{\alpha}^{(m+1)}(-\theta)=f_{\alpha}^{(m+1)}(\theta)$, $g_{\beta}^{(m+1)}(-\theta)=-g_{\beta}^{(m+1)}(\theta)$ and  $$\|\sum_{|\alpha|\geq m+2}f_{\alpha}^{(m+1)}(\theta)\rho^{\alpha}\|_{D(s'_{m+1},r'_{m+1},0,0),0}\leq Cr^{m+2},$$
$$\|\sum_{|\beta|\geq m+3}g_{\beta}^{(m+1)}(\theta)\rho^{\beta}\|_{D(s'_{m+1},r'_{m+1},0,0),0}\leq Cr^{m+3}.$$
In other words, system {\rm(\ref{fc2-1})} is changed by $\tilde{\Phi}^{(m+1)}=\tilde{\Phi}_0\circ\cdots\tilde{\Phi}_m\circ\tilde{\Phi}_{m+1}$ into {\rm(\ref{fcb2-1})}.
  \end{lem}

 \begin{proof}

 Consider a diffeomorphism $\tilde{\Phi}_{m+1}$ of the form
 \begin{equation}\label{fc2-4}
\begin{cases}
\theta=x+\sum_{|\alpha|=m+1}u_{\alpha}(x)y^{\alpha},\\
\rho=y+\sum_{|\beta|=m+2}v_{\beta}(x)y^{\beta},
\end{cases}
\end{equation}
where $u_{\alpha}$, $v_{\beta}$ will be specified. By (\ref{fc2-3}) and (\ref{fc2-4}), we have that
\begin{eqnarray}\label{fc2-6}
 \nonumber\dot{\theta}&=&\dot{x}+\sum_{|\alpha|=m+1}(\partial_{x}(u_{\alpha}(x) y^{\alpha})\cdot \dot{x})+\sum_{|\alpha|=m+1}(\partial_{y}(u_{\alpha}(x)y^{\alpha})\cdot \dot{y})\\
\nonumber&=&\omega_0+\sum_{0\leq|\alpha|\leq m}d_{\alpha}y^{\alpha}\\
\nonumber&&+\sum_{|\alpha|=m+1}(f_{\alpha}^{(m)}(x)+\omega_0\cdot \partial_{x}u_{\alpha}) y^{\alpha}\\
\nonumber&&+\sum_{|\alpha|\geq m+2}f_{\alpha}^{(m)}(x)y^{\alpha}+\partial_{x}(\sum_{|\alpha|=m+1}u_{\alpha}(x)y^{\alpha})\cdot(\sum_{0\leq|\alpha|\leq m}d_{\alpha}y^{\alpha}+\sum_{|\alpha|\geq m+1}f_{\alpha}^{(m)}(x)y^{\alpha})\\
&&+\partial_{y}(\sum_{|\alpha|=m+1}u_{\alpha}(x)y^{\alpha})\cdot \sum_{|\beta|\geq m+2}g_{\beta}^{(m)}(x)y^{\beta},
\end{eqnarray}

\begin{eqnarray}\label{fc2-7}
 \nonumber\dot{\rho}&=&\dot{y}+\sum_{|\beta|=m+2}(\partial_{x}(v_{\beta}(x) y^{\beta})\cdot \dot{x})+\sum_{|\beta|=m+2}(\partial_{y}(v_{\beta}(x) y^{\beta})\cdot\dot{y})\\
\nonumber&=&\sum_{|\beta|=m+2}(g_{\beta}^{(m)}(x)+\omega_0\cdot \partial_{x}v_{\beta}) y^{\beta}\\
\nonumber&&+\sum_{|\beta|\geq m+3}g_{\beta}^{(m)}(x)y^{\beta}+\partial_{x}(\sum_{|\beta|=m+2}v_{\beta}(x)y^{\beta})\cdot(\sum_{0\leq|\alpha|\leq m}d_{\alpha}y^{\alpha}+\sum_{|\alpha|\geq m+1}f_{\alpha}^{(m)}(x)y^{\alpha})\\
&&+\partial_{y}(\sum_{|\beta|=m+2}v_{\beta}(x)y^{\beta})\cdot\sum_{|\beta|\geq m+2}g_{\beta}^{(m)}(x)y^{\beta},
\end{eqnarray}
where $\omega_0=(\omega_{01}, \omega_{02},\cdots,\omega_{0d})^{T}$, $\omega_0\cdot \partial_{x}=\sum_{j=1}^d\omega_{0j}\partial_{x_j}$.
Noting $g_{\beta}^{(m)}(-x)=-g_{\beta}^{(m)}(x)$, $\forall|\beta|=m+2$, we have that
\begin{equation}\label{fc2-13}
\hat{g}_{\beta}^{(m)}(0)=0,\ \forall|\beta|=m+2.
\end{equation}
By (\ref{fc2-6}), (\ref{fc2-7}) and (\ref{fc2-13}), we derive homological equations:
\begin{equation}\label{fc2-14}
f_{\alpha}^{(m)}(x)+\omega_0\cdot \partial_{x}u_{\alpha}=\hat{f}_{\alpha}^{(m)}(0), \ \forall|\alpha|=m+1
\end{equation}
and
\begin{equation}\label{fc2-15}
g_{\beta}^{(m)}(x)+\omega_0\cdot \partial_{x}v_{\beta}=0, \ \forall|\beta|=m+2.
\end{equation}
By passing to Fourier coefficients, homological equation (\ref{fc2-14}) reads
\begin{equation}\label{fc2-16}
\hat{f}_{\alpha}^{(m)}(k)+i\langle k,\omega_0\rangle\hat{u}_{\alpha}(k)=0, \ \forall|\alpha|=m+1,
\end{equation}
where $k\in\mathbf{Z}^{d}/ \left\{0 \right\}$. So we can solve homological equation (\ref{fc2-14}) by setting
\begin{equation}\label{fc2-17}
u_{\alpha}(x)=-\sum_{k\in\mathbf{Z}^{d}/ \left\{0 \right\}}\frac{\hat{f}_{\alpha}^{(m)}(k)}{i\langle k,\omega_0\rangle}e^{i\langle k,x\rangle},\ \ \forall|\alpha|=m+1.
\end{equation}
 The function $f_{\alpha}^{(m)}=f_{\alpha}^{(m)}(x)$ ($\forall|\alpha|=m+1$) is analytic in $D(s'_m,0,0,0)$, so the Fourier coefficients verify
$$\|\hat{f}_{\alpha}^{(m)}(k)\|_{D(0,0,0,0),0}\leq\|f_{\alpha}^{(m)}\|_{D(s'_m,0,0,0),0}\cdot e^{-|k|s'_m}, \ \forall|\alpha|=m+1,\ \forall k\in\mathbf{Z}^{d}.$$
Then we have for $\forall|\alpha|=m+1$
\begin{eqnarray}\label{fc2-18}
 \nonumber&&\|u_{\alpha}\|_{D(s'_m-\frac{s}{20N},0,0,0),0}\\&\leq&\nonumber\sum_{k\in\mathbf{Z}^{d}/ \left\{0 \right\} }\frac{|k|^{\tau_0}\|f_{\alpha}^{(m)}\|_{D(s'_m,0,0,0),0}\cdot e^{-|k|s'_m}}{\gamma_0}e^{|k|(s'_m-\frac{s}{20N})}\\
&\leq&\nonumber\sum_{k\in\mathbf{Z}^{d}/ \left\{0 \right\} }\frac{|k|^{\tau_0}\|f_{\alpha}^{(m)}\|_{D(s'_m,0,0,0),0}\cdot e^{-\frac{s}{20N}|k|}}{\gamma_0}\\
&\leq&\nonumber\frac{C\|f_{\alpha}^{(m)}\|_{D(s'_m,0,0,0),0}}{\gamma_0(\frac{s}{20N})^{d+\tau_0}}\\
&\leq&C.
\end{eqnarray}
By (\ref{fc2-17}) and (\ref{fc2-18}), we have that
\begin{equation}\label{fb2-52}
u_{\alpha}\in C^{\omega,\infty}(D(s'_m-\frac{s}{20N},0,0,0)),\ \forall|\alpha|=m+1,
\end{equation}
\begin{equation}\label{fb2-50}
\|\sum_{|\alpha|=m+1}u_{\alpha}(x)y^{\alpha}\|_{D(s'_m-\frac{s}{20N},r'_m,0,0),0}\leq Cr^{m+1}\ll s.
\end{equation}
Similarly, we can solve homological equation (\ref{fc2-15}) by setting
\begin{equation}\label{fc2-19}
v_{\beta}(x)=-\sum_{k\in\mathbf{Z}^{d}/ \left\{0 \right\}}\frac{\hat{g}_{\beta}^{(m)}(k)}{i\langle k,\omega_0\rangle}e^{i\langle k,x\rangle},\ \forall|\beta|=m+2.
\end{equation}
And we have
\begin{equation}\label{fc2-20}
\|v_{\beta}\|_{D(s'_m-\frac{s}{20N},0,0,0),0}\leq C,\ \forall|\beta|=m+2,
\end{equation}
\begin{equation}\label{fb2-53}
v_{\beta}\in C^{\omega,\infty}(D(s'_m-\frac{s}{20N},0,0,0)),\ \forall|\beta|=m+2,
\end{equation}
\begin{equation}\label{fb2-51}
\|\sum_{|\beta|=m+2}v_{\beta}(x)y^{\beta}\|_{D(s'_m-\frac{s}{20N},r'_m,0,0),0}\leq Cr^{m+2}\ll r.
\end{equation}
By (\ref{fb2-50}), (\ref{fb2-51}) and Cauchy estimate, we have
\begin{equation}\label{fc2-21}
\|\partial_{x}(\sum_{|\alpha|=m+1}u_{\alpha}(x)y^{\alpha})\|_{D(s'_m-\frac{s}{10N},r'_m,0,0),0}\leq Cr^{m+1},
\end{equation}
\begin{equation}\label{fc2-21-1}
 \|\partial_{x}(\sum_{|\beta|=m+2}v_{\beta}(x)y^{\beta})\|_{D(s'_m-\frac{s}{10N},r'_m,0,0),0}\leq Cr^{m+2}.
\end{equation}
Let
\begin{equation}\label{fc2-29}
d_{\alpha}=\hat{f}_{\alpha}^{(m)}(0),\ \forall|\alpha|=m+1.
\end{equation}
By (\ref{fc2-6}), (\ref{fc2-7}), (\ref{fc2-14}), (\ref{fc2-15}) and (\ref{fc2-29}), we have
\begin{eqnarray}\label{fb2-1}
 \nonumber\dot{\theta}&=&\omega_0+\sum_{0\leq|\alpha|\leq m+1}d_{\alpha}y^{\alpha}\\
\nonumber&&+\sum_{|\alpha|\geq m+2}f_{\alpha}^{(m)}(x)y^{\alpha}+\partial_{x}(\sum_{|\alpha|=m+1}u_{\alpha}(x)y^{\alpha})\cdot(\sum_{0\leq|\alpha|\leq m}d_{\alpha}y^{\alpha}+\sum_{|\alpha|\geq m+1}f_{\alpha}^{(m)}(x)y^{\alpha})\\
&&+\partial_{y}(\sum_{|\alpha|=m+1}u_{\alpha}(x)y^{\alpha})\cdot \sum_{|\beta|\geq m+2}g_{\beta}^{(m)}(x)y^{\beta},
\end{eqnarray}

\begin{eqnarray}\label{fb2-2}
 \nonumber\dot{\rho}&=&\sum_{|\beta|\geq m+3}g_{\beta}^{(m)}(x)y^{\beta}+\partial_{x}(\sum_{|\beta|=m+2}v_{\beta}(x)y^{\beta})\cdot(\sum_{0\leq|\alpha|\leq m}d_{\alpha}y^{\alpha}+\sum_{|\alpha|\geq m+1}f_{\alpha}^{(m)}(x)y^{\alpha})\\
&&+\partial_{y}(\sum_{|\beta|=m+2}v_{\beta}(x)y^{\beta})\cdot \sum_{|\beta|\geq m+2}g_{\beta}^{(m)}(x)y^{\beta}.
\end{eqnarray}
By (\ref{fc2-4}), (\ref{fb2-52}), (\ref{fb2-50}), (\ref{fb2-53})-(\ref{fc2-21-1}) and by means of implicit theorem, we have that $\tilde{\Phi}_{m+1}^{-1}=\tilde{\Psi}_{m+1}$:
 \begin{equation}\label{fc2-22}
\begin{cases}
x=\theta+\sum_{|\alpha|\geq m+1}U_{\alpha}(\theta)\rho^{\alpha},\\
y=\rho+\sum_{|\beta|\geq m+2}V_{\beta}(\theta)\rho^{\beta},
\end{cases}
\end{equation}
where $(\theta,\rho)\in D(s'_m-\frac{s}{5N},r'_m-\frac{r}{5N},0,0)$, and $U_{\alpha}(\theta), V_{\beta}(\theta)\in C^{\omega,\infty}(D(s'_m-\frac{s}{5N},0,0,0))$, and
\begin{equation}\label{fc2-23}
\|\sum_{|\alpha|\geq m+1}U_{\alpha}(\theta)\rho^{\alpha}\|_{D(s'_m-\frac{s}{5N},r'_m-\frac{r}{5N},0,0),0}\leq Cr^{m+1}\ll s,
\end{equation}
\begin{equation}\label{fc2-24}
\|\sum_{|\beta|\geq m+2}V_{\beta}(\theta)\rho^{\beta}\|_{D(s'_m-\frac{s}{5N},r'_m-\frac{r}{5N},0,0),0}\leq Cr^{m+2}\ll r,
\end{equation}
\begin{eqnarray}\label{fc2-25}
 \nonumber\tilde{\Psi}_{m+1}(D(s'_{m+1},r'_{m+1},0,0))&\subset&\nonumber \tilde{\Psi}_{m+1}(D(s'_{m}-\frac{s}{5N},r'_{m}-\frac{r}{5N},0,0))\\
&\subset&D(s'_{m},r'_{m},0,0).
\end{eqnarray}

Noting $d_0=0$ and by (\ref{fc2-18}), (\ref{fc2-20}), (\ref{fc2-21}), (\ref{fc2-21-1}), (\ref{fc2-22})-(\ref{fc2-24}), we can rewrite (\ref{fb2-1}) and  (\ref{fb2-2}) as
 \begin{equation}\label{fc2-30}
\begin{cases}
\dot{\theta}=\omega_0+\sum_{0\leq|\alpha|\leq m+1}d_{\alpha}\rho^{\alpha}+\sum_{|\alpha|\geq m+2}f_{\alpha}^{(m+1)}(\theta)\rho^{\alpha},\\
\dot{\rho}=\sum_{|\beta|\geq m+3}g_{\beta}^{(m+1)}(\theta)\rho^{\beta},
\end{cases}
\end{equation}
where $d_{\alpha}$ are constants, $d_0=0$, $f_{\alpha}^{(m+1)}(\theta), g_{\beta}^{(m+1)}(\theta)\in C^{\omega,\infty}(D(s'_{m+1},0,0,0))$ and  $$\|\sum_{|\alpha|\geq m+2}f_{\alpha}^{(m+1)}(\theta)\rho^{\alpha}\|_{D(s'_{m+1},r'_{m+1},0,0),0}\leq Cr^{m+2},$$
$$\|\sum_{|\beta|\geq m+3}g_{\beta}^{(m+1)}(\theta)\rho^{\beta}\|_{D(s'_{m+1},r'_{m+1},0,0),0}\leq Cr^{m+3}.$$
Now let us consider the reversibility of the changed system. Noting $f_{\alpha}^{(m)}(-x)=f_{\alpha}^{(m)}(x)$ ($\forall|\alpha|=m+1$), $g_{\beta}^{(m)}(-x)=-g_{\beta}^{(m)}(x)$ ($\forall|\beta|=m+2$) and by (\ref{fc2-17}), (\ref{fc2-19}), we have
\begin{equation}\label{fc2-26}
u_{\alpha}(-x)=-u_{\alpha}(x),\ \ v_{\beta}(-x)=v_{\beta}(x), \ \forall|\alpha|=m+1, \ \forall|\beta|=m+2.
\end{equation}
Then by $\tilde{\Phi}_{m+1}\circ\tilde{\Psi}_{m+1}=id$, we get
\begin{equation}\label{fc2-27}
\sum_{|\alpha|=m+1}u_{\alpha}(x)y^{\alpha}+\Sigma_{|\alpha|\geq m+1}U_{\alpha}(x+\sum_{|\alpha|=m+1}u_{\alpha}(x)y^{\alpha})\cdot(y+\sum_{|\beta|=m+2}v_{\beta}(x)y^{\beta})^{\alpha}=0,
\end{equation}
\begin{equation}\label{fc2-28}
\sum_{|\beta|=m+2}v_{\beta}(x)y^{\beta}+\Sigma_{|\beta|\geq m+2}V_{\beta}(x+\sum_{|\alpha|=m+1}u_{\alpha}(x)y^{\alpha})\cdot(y+\sum_{|\beta|=m+2}v_{\beta}(x)y^{\beta})^{\beta}=0.
\end{equation}
By (\ref{fc2-26}), (\ref{fc2-27}) and (\ref{fc2-28}), we have
\begin{equation}\label{fc2-31}
U_{\alpha}(-\theta)=-U_{\alpha}(\theta),\ V_{\beta}(-\theta)=V_{\beta}(\theta),\ \forall|\alpha|\geq m+1, \ \forall|\beta|\geq m+2,
\end{equation}
where $\theta\in D(s'_m-\frac{s}{5N},0,0,0)$. It follows that
\begin{equation}\label{fc2-32}
f_{\alpha}^{(m+1)}(-\theta)=f_{\alpha}^{(m+1)}(\theta),\ g_{\beta}^{(m+1)}(-\theta)=-g_{\beta}^{(m+1)}(\theta),\ \forall|\alpha|\geq m+2, \ \forall|\beta|\geq m+3,
\end{equation}
where $\theta\in D(s'_{m+1},0,0,0$).

The proof of Lemma 2.1 is finished by (\ref{fc2-25}), (\ref{fc2-30}) and (\ref{fc2-32}).
 \end{proof}
 By Lemma \ref{lem2-1}, we have that $\tilde{\Psi}:=\tilde{\Psi}_1\circ\cdots\circ\tilde{\Psi}_N:D(\frac{s}{2},\frac{r}{2},0,0)\rightarrow D(s,r,0,0),$ and that (\ref{fc2-1}) is transformed by $\tilde{\Phi}:=\tilde{\Psi}^{-1}$ into
\begin{equation}\label{fc2-2}
\begin{cases}
\dot{\phi}=\omega_0+\sum_{1\leq|\alpha|\leq N}d_{\alpha}\mu^{\alpha}+\sum_{|\alpha|\geq N+1}f_{\alpha}^{(N)}(\phi)\mu^{\alpha},\\
\dot{\mu}=\sum_{|\beta|\geq N+2}g_{\beta}^{(N)}(\phi)\mu^{\beta},
\end{cases}
\end{equation}
where $d_{\alpha}$ are constants, $f_{\alpha}^{(N)}(\phi), g_{\beta}^{(N)}(\phi)\in C^{\omega,\infty}(D(\frac{s}{2},0,0,0))$, $f_{\alpha}^{(N)}(-\phi)=f_{\alpha}^{(N)}(\phi)$, $g_{\beta}^{(N)}(-\phi)=-g_{\beta}^{(N)}(\phi)$ and
$$\|\sum_{|\alpha|\geq N+1}f_{\alpha}^{(N)}(\phi)\mu^{\alpha}\|_{D(\frac{s}{2},\frac{r}{2},0,0),0}\leq Cr^{N+1},\ \ \  \|\sum_{|\beta|\geq N+2}g_{\beta}^{(N)}(\phi)\mu^{\beta}\|_{D(\frac{s}{2},\frac{r}{2},0,0),0}\leq Cr^{N+2}.$$

Let $\frac{s}{2}=s_0$, $\frac{r}{2}=2a$. And take $b>0$. For example, we can let $b=1$. Now we introduce parameter vector $\xi\in B(a)$ by
$$\mu=\xi+y, \ \xi\in B(a), \ y\in B_{\mathbf{C}}(r_0), \ r_0=a.$$
By replacing $\phi$ with $x$, we can rewrite (\ref{fc2-2}) as
\begin{equation}\label{fc2-33}
\begin{cases}
\dot{x}=\omega_0+\Lambda_{(00)}(\xi)+f_0(x,\xi)+\sum_{|\alpha|\geq1}f_{\alpha}(x,\xi)y^{\alpha},\\
\dot{y}=g_0(x,\xi)+g_1(x,\xi)\cdot y+\sum_{|\beta|\geq2}g_{\beta}(x,\xi)y^{\beta},
\end{cases}
\end{equation}
where $(x,y,\xi)\in D(s_0,r_0,a,b)$, $g_1(x,\xi)\cdot y:=\sum_{|\beta|=1}g_{\beta}(x,\xi)y^{\beta}$ (In the following, the definitions are similar, for example, $f_1(x,\xi)\cdot y:=\sum_{|\alpha|=1}f_{\alpha}(x,\xi)y^{\alpha}$ and so on), and $\Lambda_{(00)}(\xi)$, $f_0(x,\xi)$, $g_0(x,\xi)$, $g_1(x,\xi)$, $f_{\alpha}(x,\xi)$ ($|\alpha|\geq1$), $g_{\beta}(x,\xi)$ ($|\beta|\geq2$) satisfy the following lemma:
\begin{lem}\label{lem2-2} (0) The function $\Lambda_{(00)}(\xi)\in C^{\omega,\infty}(D(0,0,a,0))$ and
$$\Lambda_{(00)}=\sum_{1\leq|\alpha|\leq N}d_{\alpha}\xi^{\alpha}, \Lambda_{(00)}=O^1(\xi), \ \|\Lambda_{(00)}(\xi)\|_{D(0,0,a,0),\ell}\leq Ca;$$
(1) The function $f_0(x,\xi)\in C^{\omega,\infty}(D(s_0,0,a,0))$, $f_0(-x,\xi)=f_0(x,\xi)$ and
$$f_0(x,\xi)=O^{N+1}(\xi), \ \|f_0(x,\xi)\|_{D(s_0,0,a,0),\ell}\leq Ca^{N+1};$$
(2) The function $g_0(x,\xi)\in C^{\omega,\infty}(D(s_0,0,a,0))$, $g_0(-x,\xi)=-g_0(x,\xi)$ and
$$g_0(x,\xi)=O^{N+1}(\xi),\ \|g_0(x,\xi)\|_{D(s_0,0,a,0),\ell}\leq Ca^{N+2};$$
(3) The function $g_1(x,\xi)\in C^{\omega,\infty}(D(s_0,0,a,0))$, $g_1(-x,\xi)=-g_1(x,\xi)$ and
$$g_1(x,\xi)=O^{N+1}(\xi),\ \|g_1(x,\xi)\|_{D(s_0,0,a,0),\ell}\leq Ca^{N+1};$$
(4) The function $f_{\alpha}(x,\xi)\in C^{\omega,\infty}(D(s_0,0,a,0))$, $f_{\alpha}(-x,\xi)=f_{\alpha}(x,\xi)$, $\forall|\alpha|\geq1$ and
$$\|\sum_{|\alpha|\geq1}f_{\alpha}(x,\xi)y^{\alpha}\|_{D(s_0,r_0,a,0),\ell}\leq Cr_0.$$
(5) The function $g_{\beta}(x,\xi)\in C^{\omega,\infty}(D(s_0,0,a,0))$, $g_{\beta}(-x,\xi)=-g_{\beta}(x,\xi)$, $\forall|\beta|\geq2$ and
$$\|\sum_{|\beta|\geq2}g_{\beta}(x,\xi)y^{\beta}\|_{D(s_0,r_0,a,0),\ell}\leq Cr_0^{N+2}\leq Cr_0^2.$$
 \end{lem}

 Similar to Section 3.1 in \cite{a2}, there exists the BNF of (\ref{fc1-3}):
 \begin{lem}\label{lem2-3} There exists $$\Psi:x=\phi+\sum_{|\alpha|\geq1}\bar{u}_{\alpha}(\phi)\mu^{\alpha},\
y=\mu+\sum_{|\beta|\geq2}\bar{v}_{\beta}(\phi)\mu^{\beta},$$
where $\bar{u}_{\alpha}(\phi), \bar{v}_{\beta}(\phi)\in C^{\omega,\infty}(D(s,0,0,0))$ such that system {\rm(\ref{fc1-3})} is changed by $\Phi=\Psi^{-1}$ into
 \begin{equation}\label{fb2-100}
\begin{cases}
\dot{\phi}=f_F(\mu),\\
\dot{\mu}=0,
\end{cases}
\end{equation}
where $f_F=\omega_0+\sum_{|\alpha|\geq1}d_{\alpha}\mu^{\alpha}$, $d_{\alpha}$ are constants. Moreover, any $d_{\alpha}$ is unique.
\end{lem}
\begin{rem}\label{rem2-4} The unique (\ref{fb2-100}) is the BNF of (\ref{fc1-3}) around $\Gamma_0$. Not considering the convergence of $\sum_{|\alpha|\geq1}\bar{u}_{\alpha}(\phi)\mu^{\alpha}$, $\sum_{|\beta|\geq2}\bar{v}_{\beta}(\phi)\mu^{\beta}$ and $f_F(\mu)$, we can prove Lemma \ref{lem2-3} easily by using the method of solving (\ref{fc2-14}) and (\ref{fc2-15}). By (\ref{fc2-17}) and (\ref{fc2-19}), we have that $\bar{u}_{\alpha}(-\phi)=-\bar{u}_{\alpha}(\phi)$, $\bar{v}_{\beta}(-\phi)=\bar{v}_{\beta}(\phi)$, $\hat{\bar{v}}_{\beta}(0)=0$, $\forall|\alpha|\geq1$, $\forall|\beta|\geq2$. Moreover, $d_{\alpha}$ ($|\alpha|\leq N$) in Lemma \ref{lem2-1} are the same as them in Lemma \ref{lem2-3}.
\end{rem}
\section{Flat function and Flat operator}\label{sec3}
In this section, we introduce the flat function and flat operator, which will be used in Section 5.
\begin{defn}\label{defn3-1} \cite{a2}   We call that $f\in C^{\omega,\infty}(D(s,r,a,b))$ is a flat function on $DC(\gamma,\tau)$, if
$$\partial_{x}^{\alpha}\partial_{y}^{\beta}\partial_{\xi}^{p}\partial_{\eta}^{q}f(x,y,\xi,\eta)=0, \ (x,y,\xi,\eta)\in D(s,r,a,b)$$
for all multi-indices $\alpha$, $\beta$, $p$, $q$ whenever $\eta\in DC(\gamma,\tau)$. For simplicity, we call $f$ is $(\gamma,\tau)$-flat when $f$ is a flat function on $DC(\gamma,\tau)$.
\end{defn}

Let $\tilde{l}: \mathbf{R}\rightarrow\mathbf{R}$ denote a fixed non-negative $C^{\infty}$ function such that $|\tilde{l}|\leq1$, and $\tilde{l}=0$ if $|x|\geq\frac{1}{2}$, and $\tilde{l}=1$ if $|x|\leq\frac{1}{4}$. As in \cite{a2}, for $0<\gamma<1$, $\tau>d-1$ we define a cut-off operator (flat operator) $P_{\gamma,\tau}$
\begin{equation}\label{fb3-20}
(P_{\gamma,\tau}f)(x,y,\xi,\eta)=\sum_{k\in\mathbf{Z}^{d}/ \left\{0 \right\}}\hat{f}(k,y,\xi,\eta)e^{i\langle k,x\rangle}\tilde{l}(\langle k,\eta\rangle\frac{|k|^{\tau}}{\gamma})
\end{equation}
for any $f\in C^{\omega,\infty}(D(s,r,a,b))$.
\begin{lem}\label{lem3-2} For any $s'<s$, we have
$$\|P_{\gamma,\tau}f\|_{D(s',r,a,b),\ell}\leq \frac{C_{\ell}}{\gamma^{\ell}}(\frac{1}{s-s'})^{(\tau+1)\ell+d}\|f\|_{D(s,r,a,b),\ell},$$
where $C_{\ell}$ is a constant depending on only $\ell$, $\tau$, $d$ and $\tilde{l}$.
\end{lem}
 \begin{proof} The proof is essentially the same as that of \cite{a2}. The function $f=f(x,y,\xi,\eta)$ is analytic in $D(s,r,a,b)$, so the Fourier coefficients verify
 $$\|\hat{f}(k,y,\xi,\eta)\|_{D(0,r,a,b),\ell}\leq\|f\|_{D(s,r,a,b),\ell}e^{-|k|s}, \ \forall k\in\mathbf{Z}^{d}.$$
 The functions
 $$\tilde{l}_{k}(\eta):=\tilde{l}(\langle k,\eta\rangle\frac{|k|^{\tau}}{\gamma})$$
 verify
 $$\|\tilde{l}_{k}\|_{D(0,0,0,b),\ell}\leq\frac{|k|^{(\tau+1)\ell}}{\gamma^{\ell}}\|\tilde{l}\|_{B(\infty),\ell},$$
where $\|\tilde{l}\|_{B(\infty),\ell}:=\sup_{x\in \mathbf{R},p\leq\ell}|\partial_{x}^{p}\tilde{l}(x)|$ which is a constant depending only on $\tilde{l}$ and $\ell$. Thus
\begin{eqnarray}
 \nonumber&&\|P_{\gamma,\tau}f\|_{D(s',r,a,b),\ell}\\&\leq&\nonumber C_{\ell}\sum_{k\in\mathbf{Z}^{d}/ \left\{0 \right\} }e^{|k|s'}(\|\hat{f}(k,y,\xi,\eta)\|_{D(0,r,a,b),\ell}\|\tilde{l}_{k}\|_{D(0,0,0,b),0}\\
&&+\nonumber\|\hat{f}(k,y,\xi,\eta)\|_{D(0,r,a,b),0}\|\tilde{l}_{k}\|_{D(0,0,0,b),\ell})\\
&\leq&\nonumber C_{\ell}\sum_{k\in\mathbf{Z}^{d}/ \left\{0 \right\} }e^{|k|s'}\frac{|k|^{(\tau+1)\ell}}{\gamma^{\ell}}\|\tilde{l}\|_{B(\infty),\ell}\|f\|_{D(s,r,a,b),\ell}e^{-|k|s}\\
&\leq&\nonumber \frac{C_{\ell}}{\gamma^{\ell}}\|f\|_{D(s,r,a,b),\ell}\sum_{k\in\mathbf{Z}^{d}/ \left\{0 \right\} }|k|^{(\tau+1)\ell} e^{-|k|(s-s')}\\
&\leq&\nonumber \frac{C_{\ell}}{\gamma^{\ell}}(\frac{1}{s-s'})^{(\tau+1)\ell+d}\|f\|_{D(s,r,a,b),\ell}.
\end{eqnarray}
  \end{proof}

\section{Iterative Lemma}\label{sec4}
\subsection{Initial data}
Let us return to (\ref{fc2-33}). By conclusion (2) of Lemma \ref{lem2-2}, we have
\begin{equation}\label{fb4-40}
\hat{g}_0(0,\xi)=0.
\end{equation}
It follows that the solution of (\ref{fb4-41}) exists (by solving (\ref{fc5-5}) with $m=0$).
 Let $v_1(x,\xi,\eta)=\mathcal{L}g_0(x,\xi)$ be the solution of (\ref{fb4-41}).
\begin{equation}\label{fb4-41}
(1-P_{\gamma,\tau})g_0(x,\xi)+\eta\cdot\partial_{x}v_1=0,\ \eta\in B(b),
\end{equation}
where $\eta=(\eta_1,\eta_2,\cdots,\eta_d)^{T}$, $\eta\cdot \partial_{x}=\sum_{j=1}^d\eta_j\partial_{x_j}$.

Let
\begin{equation}\label{fb4-42}
\Lambda_{(0)}=\Lambda_{(00)}+\bar{\Lambda}_{(0)},
\end{equation}
where
\begin{equation}\label{fb4-43}
\bar{\Lambda}_{(0)}(\xi,\eta)=(f_0-f_1\cdot \mathcal{L}g_0)^{\hat{}}(0,\xi,\eta).
\end{equation}
Let $\varepsilon_0=a^{N+1}$. By (\ref{fb4-43}) and (\ref{fc5-12}), (\ref{fb5-4}) with $m=0$ and by conclusions (1), (4) of Lemma \ref{lem2-2}, we have
\begin{equation}\label{fb4-44}
\bar{\Lambda}_{(0)}(\xi,\eta)\in C^{\omega,\infty}(D(0,0,a,b)),\ \bar{\Lambda}_{(0)}(\xi,\eta)=O^{N+1}(\xi),\ \|\bar{\Lambda}_{(0)}(\xi,\eta)\|_{D(0,0,a,b),\ell}\leq \frac{C\varepsilon_0}{\gamma^{\ell+1}}.
\end{equation}

Let $$f_{0}^{(0)}(x,\xi,\eta)=f_0(x,\xi)-\bar{\Lambda}_{(0)}(\xi,\eta), \ g_{0}^{(0)}(x,\xi,\eta)=g_0(x,\xi),\ g_{1}^{(0)}(x,\xi,\eta)=g_1(x,\xi),$$
 $$f_{\alpha}^{(0)}(x,\xi,\eta)=f_{\alpha}(x,\xi), \forall|\alpha|\geq1, \ \ \ \ g_{\beta}^{(0)}(x,\xi,\eta)=g_{\beta}(x,\xi),  \forall|\beta|\geq2,$$ $h_{(0,0)}=\mathbf{1}_{d\times d}$ be $d\times d$ unit matrix, $H_{(0,0)}=\mathbf{O}_{d\times d}$ be $d\times d$ zero matrix, $R^{(1)}_{(0,0)}=R^{(2)}_{(0,0)}=0. $

By replacing $(x,y)$ with $(\tilde{x},\tilde{y})$, we can rewrite (\ref{fc2-33}) as
\begin{equation}\label{fc4-1}
\begin{cases}
\dot{\tilde{x}}=\eta+h_{(0,0)}\cdot(\omega_0-\eta+\Lambda_{(0)}(\xi,\eta))+f_{0}^{(0)}(\tilde{x},\xi,\eta)+\sum_{|\alpha|\geq1}f_{\alpha}^{(0)}(\tilde{x},\xi,\eta)\tilde{y}^{\alpha}+R^{(1)}_{(0,0)},\\
\dot{\tilde{y}}=H_{(0,0)}\cdot(\omega_0-\eta+\Lambda_{(0)}(\xi,\eta))+g_{0}^{(0)}(\tilde{x},\xi,\eta)+g_{1}^{(0)}(\tilde{x},\xi,\eta) \tilde{y}+\sum_{|\beta|\geq2}g_{\beta}^{(0)}(\tilde{x},\xi,\eta)\tilde{y}^{\beta} \\ \ \ \  \ \
+R^{(2)}_{(0,0)},
\end{cases}
\end{equation}
where $\Lambda_{(0)}$, $h_{(0,0)}$, $H_{(0,0)}$, $f_{0}^{(0)}$, $f_{\alpha}^{(0)}$ ($|\alpha|\geq1$), $g_{0}^{(0)}$, $g_{1}^{(0)}$, $g_{\beta}^{(0)}$ ($|\beta|\geq2$), $R^{(1)}_{(0,0)}$ and $R^{(2)}_{(0,0)}$ fulfill conditions {\rm(\ref{fc4-3})}-{\rm(\ref{fb4-13})} with $m=0$.
\subsection{Iterative constants and domains}
$\bullet$ $m$ ($m=0,1,2,\cdots$) denotes the step of Newton iteration;

$\bullet$ $e_m=\frac{\sum_{l=1}^{m}l^{-2}}{2\sum_{l=1}^{\infty}l^{-2}}$, $m=1,2,\cdots$;

$\bullet$ $s_0=\frac{s}{2}>0$, $s_m=s_0(1-e_m)\geq\frac{s_0}{2}$, $m=1,2,\cdots$;

$\bullet$ $s^{(i)}_m=s_m-\frac{i}{10}(s_m-s_{m+1})$, $i=1,2,3,4,5,6$, $m=0,1,2,\cdots$;

$\bullet$ $r_0=a$, $r_m=r_0(1-e_m)\geq\frac{r_0}{2}$, $m=1,2,\cdots$;

$\bullet$ $r^{(1)}_m=r_m-\frac{1}{10}(r_m-r_{m+1})$, $m=0,1,2,\cdots$;

$\bullet$ $\varepsilon_0=a^{N+1}$, $\varepsilon_m=\varepsilon_0^{(\frac{4}{3})^m}$, $m=1,2,\cdots$;

$\bullet$ the operator $\mathcal{L}$ is defined by that $\mathcal{L}g(x,\xi,\eta)=v$ is the solution of equation $(1-P_{\gamma,\tau})g(x,\xi,\eta)+\eta\cdot\partial_{x}v=0$.

\subsection{Iterative Lemma}
Let $0<a^{\frac{N}{20(\ell+1)}}<\gamma<1$, $\tau>d-1$. Then we give the iterative lemma.
\begin{lem}\label{lem4-1}  Suppose that we have had $m+1$ {\rm(}$m=0,1,2,\cdots${\rm)} diffeomorphisms $\Phi_0=\Psi_0^{-1}=id$, $\Phi_1=\Psi_1^{-1}$, $\cdots$, $\Phi_m=\Psi_m^{-1}$ with
$$\Psi_j:D(s_j,r_j,a,b)\rightarrow D(s_{j-1},r_{j-1},a,b), \ j=1,2,\cdots,m$$
of the form
$$\Phi_j:x=\tilde{x}+u_j(\tilde{x},\xi,\eta), \ y=\tilde{y}+v_j(\tilde{x},\xi,\eta)+w_j(\tilde{x},\xi,\eta)\tilde{y},\ j=1,2,\cdots,m,$$
where $u_j, v_j, w_j\in C^{\omega,\infty}(D(s_j,0,a,b))$ and
$$\|u_j\|_{D(s_j,0,a,b),\ell},\ \|v_j\|_{D(s_j,0,a,b),\ell},\ \|w_j\|_{D(s_j,0,a,b),\ell}\leq \frac{C\varepsilon_{j-1}}{\gamma^{2\ell+2}}, \ j=1,2,\cdots,m$$
such that system {\rm(\ref{fc4-1})} is changed by $\Phi^{(m)}=\Phi_0\circ\cdots\Phi_m$ into
\begin{equation}\label{fc4-2}
\begin{cases}
\dot{x}=\eta+h_{(m)}\cdot(\omega_0-\eta+\Lambda_{(m)}(\xi,\eta))+f_{0}^{(m)}(x,\xi,\eta)+\sum_{|\alpha|\geq1}f_{\alpha}^{(m)}(x,\xi,\eta)y^{\alpha}+R^{(1)}_{(m)},\\
\dot{y}=H_{(m)}\cdot(\omega_0-\eta+\Lambda_{(m)}(\xi,\eta))+g_{0}^{(m)}(x,\xi,\eta)+g_{1}^{(m)}(x,\xi,\eta)y\\ \ \ \  \ \
+\sum_{|\beta|\geq2}g_{\beta}^{(m)}(x,\xi,\eta)y^{\beta}
+R^{(2)}_{(m)},
\end{cases}
\end{equation}
where
\begin{equation}\label{fc4-3}
\Lambda_{(j)}=\Lambda_{(00)}+\bar{\Lambda}_{(j)},\ \bar{\Lambda}_{(j)}\in C^{\omega,\infty}(D(0,0,a,b)),\ \bar{\Lambda}_{(j)}(\xi,\eta)=O^{N+1}(\xi),\ j=0,1,2,\cdots,m,
\end{equation}
\begin{equation}\label{fc4-4}
\|\bar{\Lambda}_{(0)}\|_{D(0,0,a,b),\ell}\leq \frac{C\varepsilon_0}{\gamma^{\ell+1}}, \ \|\bar{\Lambda}_{(j)}-\bar{\Lambda}_{(j-1)}\|_{D(0,0,a,b),\ell}\leq \frac{C\varepsilon_j}{\gamma^{\ell+1}},\ j=1,2,\cdots,m;
\end{equation}
\begin{equation}\label{fc4-5}
h_{(m)}=\sum_{j=0}^mh_{(j,m)}(x,\xi,\eta),\ h_{(0,m)}=\mathbf{1}_{d\times d},\ h_{(j,m)}\in C^{\omega,\infty}(D(s_m,0,a,b)),
\end{equation}
\begin{equation}\label{fc4-6}
h_{(j,m)}(-x,\xi,\eta)=h_{(j,m)}(x,\xi,\eta),
\end{equation}
\begin{equation}\label{fc4-7}
\|h_{(j,m)}\|_{D(s_m,0,a,b),\ell}\leq \frac{C\varepsilon_{j-1}}{\gamma^{2\ell+2}},\ j=1,2,\cdots,m;
\end{equation}
\begin{equation}\label{fc4-8}
H_{(m)}=\sum_{j=0}^mH_{(j,m)}(x,y,\xi,\eta),\ H_{(0,m)}=\mathbf{O}_{d\times d},\ H_{(j,m)}\in C^{\omega,\infty}(D(s_m,r_m,a,b)),
\end{equation}
\begin{equation}\label{fc4-9}
H_{(j,m)}(-x,y,\xi,\eta)=-H_{(j,m)}(x,y,\xi,\eta),
\end{equation}
\begin{equation}\label{fc4-10}
\|H_{(j,m)}\|_{D(s_m,r_m,a,b),\ell}\leq \frac{C\varepsilon_{j-1}}{\gamma^{2\ell+2}},\ j=1,2,\cdots,m;
\end{equation}
\begin{equation}\label{fc4-11}
f_{0}^{(m)}\in C^{\omega,\infty}(D(s_m,0,a,b)),\ f_{0}^{(m)}=O^{N+1}(\xi),\ \|f_{0}^{(m)}\|_{D(s_m,0,a,b),\ell}\leq C\varepsilon_m,\
\end{equation}
\begin{equation}\label{fc4-12}
f_{0}^{(m)}(-x,\xi,\eta)=f_{0}^{(m)}(x,\xi,\eta);
\end{equation}
\begin{equation}\label{fc4-13}
g_{0}^{(m)}\in C^{\omega,\infty}(D(s_m,0,a,b)),\ g_{0}^{(m)}=O^{N+1}(\xi),\ \|g_{0}^{(m)}\|_{D(s_m,0,a,b),\ell}\leq C\varepsilon_m,\
\end{equation}
\begin{equation}\label{fc4-14}
 g_{0}^{(m)}(-x,\xi,\eta)=-g_{0}^{(m)}(x,\xi,\eta);
\end{equation}
\begin{equation}\label{fc4-15}
g_{1}^{(m)}\in C^{\omega,\infty}(D(s_m,0,a,b)),\ g_{1}^{(m)}=O^{N+1}(\xi),\ \|g_{1}^{(m)}\|_{D(s_m,0,a,b),\ell}\leq C\varepsilon_m,\
\end{equation}
\begin{equation}\label{fc4-16}
 g_{1}^{(m)}(-x,\xi,\eta)=-g_{1}^{(m)}(x,\xi,\eta);
\end{equation}
\begin{equation}\label{fc4-17}
f_{\alpha}^{(m)}\in C^{\omega,\infty}(D(s_m,0,a,b)),\ \|\sum_{|\alpha|\geq1}f_{\alpha}^{(m)}(x,\xi,\eta)y^{\alpha}\|_{D(s_m,r_m,a,b),\ell}\leq Cr_m,
\end{equation}
\begin{equation}\label{fc4-18}
 f_{\alpha}^{(m)}(-x,\xi,\eta)=f_{\alpha}^{(m)}(x,\xi,\eta),\ \forall|\alpha|\geq1;
\end{equation}
\begin{equation}\label{fc4-19}
g_{\beta}^{(m)}\in C^{\omega,\infty}(D(s_m,0,a,b)),\ \|\sum_{|\beta|\geq2}g_{\beta}^{(m)}(x,\xi,\eta)y^{\beta}\|_{D(s_m,r_m,a,b),\ell}\leq Cr_m^2,
\end{equation}
\begin{equation}\label{fc4-20}
 g_{\beta}^{(m)}(-x,\xi,\eta)=-g_{\beta}^{(m)}(x,\xi,\eta),\ \forall|\beta|\geq2;
\end{equation}
\begin{equation}\label{fc4-21}
(f_{0}^{(m)}-f_{1}^{(m)}\cdot \mathcal{L}g_{0}^{(m)})^{\hat{}}(0,\xi,\eta)=0;
\end{equation}
\begin{equation}\label{fc4-22}
R_{(m)}^{(1)}=\sum_{j=0}^mR_{(j,m)}^{(1)}(x,\xi,\eta),\ R^{(1)}_{(0,m)}=0,
\end{equation}
\begin{equation}\label{fc4-23}
R_{(j,m)}^{(1)}(-x,\xi,\eta)=R_{(j,m)}^{(1)}(x,\xi,\eta),
\end{equation}
\begin{equation}\label{fc4-24}
R_{(j,m)}^{(1)}\in C^{\omega,\infty}(D(s_m,0,a,b)) \ is \ flat \ on \ DC(\gamma,\tau),
\end{equation}
\begin{equation}\label{fc4-25}
\|R_{(j,m)}^{(1)}\|_{D(s_m,0,a,b),\ell}\leq \frac{C\varepsilon_{j-1}}{\gamma^{2\ell+2}},\ j=1,2,\cdots,m;
\end{equation}
\begin{equation}\label{fb4-10}
 R_{(m)}^{(2)}=\sum_{j=0}^mR_{(j,m)}^{(2)}(x,y,\xi,\eta),\ R^{(2)}_{(0,m)}=0,
\end{equation}
\begin{equation}\label{fb4-11}
 R_{(j,m)}^{(2)}(-x,y,\xi,\eta)=-R_{(j,m)}^{(2)}(x,y,\xi,\eta),
\end{equation}
\begin{equation}\label{fb4-12}
R_{(j,m)}^{(2)}\in C^{\omega,\infty}(D(s_m,r_m,a,b)) \ is \ flat \ on \ DC(\gamma,\tau),
\end{equation}
\begin{equation}\label{fb4-13}
\|R_{(j,m)}^{(2)}\|_{D(s_m,r_m,a,b),\ell}\leq \frac{C\varepsilon_{j-1}}{\gamma^{2\ell+2}},\ j=1,2,\cdots,m.
\end{equation}
Then there is a diffeomorphism $\Phi_{m+1}=\Psi_{m+1}^{-1}$  with
$$\Psi_{m+1}:D(s_{m+1},r_{m+1},a,b)\rightarrow D(s_m,r_m,a,b)$$
of the form
$$\Phi_{m+1}:\theta=x+u_{m+1}(x,\xi,\eta), \ \rho=y+v_{m+1}(x,\xi,\eta)+w_{m+1}(x,\xi,\eta) y,$$
where $u_{m+1}, v_{m+1}, w_{m+1}\in C^{\omega,\infty}(D(s_{m+1},0,a,b))$ and
$$\|u_{m+1}\|_{D(s_{m+1},0,a,b),\ell},\ \|v_{m+1}\|_{D(s_{m+1},0,a,b),\ell},\ \|w_{m+1}\|_{D(s_{m+1},0,a,b),\ell}\leq \frac{C\varepsilon_m}{\gamma^{2\ell+2}}$$
such that system {\rm(\ref{fc4-2})} is changed by $\Phi_{m+1}$ into
\begin{equation}\label{fc4-26}
\begin{cases}
\dot{\theta}=\eta+h_{(m+1)}\cdot(\omega_0-\eta+\Lambda_{(m+1)}(\xi,\eta))+f_{0}^{(m+1)}(\theta,\xi,\eta)+\Sigma_{|\alpha|\geq1}f_{\alpha}^{(m+1)}(\theta,\xi,\eta)\rho^{\alpha}\\ \ \ \ \ \
+R^{(1)}_{(m+1)},\\
\dot{\rho}=H_{(m+1)}\cdot(\omega_0-\eta+\Lambda_{(m+1)}(\xi,\eta))+g_{0}^{(m+1)}(\theta,\xi,\eta)+g_{1}^{(m+1)}(\theta,\xi,\eta) \rho\\ \ \ \ \ \
 +\sum_{|\beta|\geq2}g_{\beta}^{(m+1)}(\theta,\xi,\eta)\rho^{\beta}+R^{(2)}_{(m+1)},
\end{cases}
\end{equation}
where $\Lambda_{(m+1)}$, $h_{(m+1)}$, $H_{(m+1)}$, $f_{0}^{(m+1)}$, $g_{0}^{(m+1)}$, $g_{1}^{(m+1)}$, $f_{\alpha}^{(m+1)}$ {\rm(}$|\alpha|\geq1${\rm)}, $g_{\beta}^{(m+1)}$ {\rm(}$|\beta|\geq2${\rm)}, $R^{(1)}_{(m+1)}$, $R^{(2)}_{(m+1)}$ satisfy all conditions {\rm(\ref{fc4-3})}-{\rm(\ref{fb4-13})} with replacing $m$ by $m+1$. In other words, system {\rm(\ref{fc4-1})} is changed by $\Phi^{(m+1)}=\Phi_0\circ\cdots\circ\Phi_m\circ\Phi_{m+1}$ into {\rm(\ref{fc4-26})}.
  \end{lem}

\section{Proof of the iterative lemma}\label{sec5}
In this section, we will prove the iterative lemma.
\subsection{Derivation of homological equations}
 Consider a diffeomorphism $\Phi_{m+1}$ of the form
 \begin{equation}\label{fc5-1}
\begin{cases}
\theta=x+u_{m+1}(x,\xi,\eta),\\
\rho=y+v_{m+1}(x,\xi,\eta)+w_{m+1}(x,\xi,\eta) y,
\end{cases}
\end{equation}
where $u_{m+1}$, $v_{m+1}$, $w_{m+1}$ will be specified. By (\ref{fc4-2}) and (\ref{fc5-1}), we have that
\begin{eqnarray}\label{fc5-2}
 \nonumber\dot{\theta}&=&\dot{x}+\partial_{x}u_{m+1}\cdot \dot{x}\\
\nonumber&=&\eta+(h_{(m)}+\partial_{x}u_{m+1}\cdot h_{(m)})\cdot (\omega_0-\eta+\Lambda_{(m)}(\xi,\eta))\\
\nonumber&+&f_{0}^{(m)}-f_{1}^{(m)}\cdot v_{m+1}+\eta\cdot \partial_{x}u_{m+1}\\
\nonumber&+&f_{1}^{(m)}\cdot (y+v_{m+1})+\sum_{|\alpha|\geq2}f_{\alpha}^{(m)}(x,\xi,\eta)y^{\alpha}\\
&+&\partial_{x}u_{m+1}\cdot(f_{0}^{(m)}+\sum_{|\alpha|\geq1}f_{\alpha}^{(m)}(x,\xi,\eta)y^{\alpha})+R_{(m)}^{(1)}+\partial_{x}u_{m+1}\cdot R_{(m)}^{(1)},
\end{eqnarray}

\begin{eqnarray}\label{fc5-3}
 \nonumber\dot{\rho}&=&\dot{y}+\partial_{x}v_{m+1}\cdot\dot{x}+\partial_{x}(w_{m+1}\cdot y)\cdot\dot{x}+w_{m+1}\cdot\dot{y} \\
 \nonumber&=&(H_{(m)}+\partial_{x}v_{m+1}\cdot h_{(m)}+\partial_{x}(w_{m+1}\cdot y)\cdot h_{(m)}+w_{m+1}\cdot H_{(m)})\cdot (\omega_0-\eta\\
\nonumber&&+\Lambda_{(m)}(\xi,\eta))+g_{0}^{(m)}+\eta\cdot \partial_{x}v_{m+1}\\
\nonumber&&+(g_{1}^{(m)}-2g_{2}^{(m)}\cdot v_{m+1}+\partial_{x}v_{m+1}\cdot f_{1}^{(m)}+ \eta\cdot \partial_{x}w_{m+1}) y\\
\nonumber&&+\sum_{|\beta|=2}g_{\beta}^{(m)}(y+v_{m+1})^{\beta}-\sum_{|\beta|=2}g_{\beta}^{(m)}v^{\beta}_{m+1}+\sum_{|\beta|\geq3}g_{\beta}^{(m)}(x,\xi,\eta)y^{\beta}\\
\nonumber&&+\partial_{x}v_{m+1}\cdot(f_{0}^{(m)}+\sum_{|\alpha|\geq2}f_{\alpha}^{(m)}(x,\xi,\eta)y^{\alpha})\\
\nonumber&&+\partial_{x}(w_{m+1}\cdot y)\cdot(f_{0}^{(m)}+\sum_{|\alpha|\geq1}f_{\alpha}^{(m)}(x,\xi,\eta)y^{\alpha})\\
\nonumber&&+w_{m+1}\cdot (g_{0}^{(m)}+g_{1}^{(m)}\cdot y+\sum_{|\beta|\geq2}g_{\beta}^{(m)}(x,\xi,\eta)y^{\beta})\\
&&+R_{(m)}^{(2)}+\partial_{x}v_{m+1}\cdot R_{(m)}^{(1)}+\partial_{x}(w_{m+1}\cdot y)\cdot R_{(m)}^{(1)}+w_{m+1}\cdot R_{(m)}^{(2)},
\end{eqnarray}
where $\eta=(\eta_1,\eta_2,\cdots,\eta_d)^{T}$, $\eta\cdot \partial_{x}=\sum_{j=1}^d\eta_j\partial_{x_j}$, $(2g_{2}^{(m)}\cdot v_{m+1})y:=\sum_{|\beta|=2}g_{\beta}^{(m)}(y+v_{m+1})^{\beta}-\sum_{|\beta|=2}g_{\beta}^{(m)}y^{\beta}-\sum_{|\beta|=2}g_{\beta}^{(m)}v^{\beta}_{m+1}$.

By (\ref{fc5-2}) and (\ref{fc5-3}), we derive homological equations:
\begin{equation}\label{fc5-4}
(1-P_{\gamma,\tau})(f_{0}^{(m)}-f_{1}^{(m)}\cdot v_{m+1})+\eta\cdot \partial_{x}u_{m+1}=0,
\end{equation}
\begin{equation}\label{fc5-5}
(1-P_{\gamma,\tau})g_{0}^{(m)}+\eta\cdot \partial_{x}v_{m+1}=0,
\end{equation}
\begin{equation}\label{fc5-6}
(1-P_{\gamma,\tau})(g_{1}^{(m)}-2g_{2}^{(m)}\cdot v_{m+1}+\partial_{x}v_{m+1}\cdot f_{1}^{(m)})+\eta\cdot \partial_{x}w_{m+1}=0.
\end{equation}

\subsection{Solution to the homological equations}
$\bullet$ Solution to (\ref{fc5-5}).

By (\ref{fc4-14}), we have
\begin{equation}\label{fc5-7}
 \hat{g}^{(m)}_{0}(0,\xi,\eta)=0.
\end{equation}
Then by passing to Fourier coefficients, homological equation (\ref{fc5-5}) reads
\begin{equation}\label{fc5-8}
(1-\tilde{l}(\langle k,\eta\rangle\frac{|k|^{\tau}}{\gamma}))\hat{g}_{0}^{(m)}(k,\xi,\eta)+i\langle k,\eta\rangle\hat{v}_{m+1}(k,\xi,\eta)=0,
\end{equation}
where $k\in\mathbf{Z}^{d}/ \left\{0 \right\}$.

 If $|\langle k,\eta\rangle\frac{|k|^{\tau}}{\gamma}|\leq\frac{1}{4}$, we have $1-\tilde{l}(\langle k,\eta\rangle\frac{|k|^{\tau}}{\gamma})=0$, so at this time, we take
\begin{equation}\label{fb5-3}
\hat{v}_{m+1}(k,\xi,\eta)=0.
\end{equation}
If $|\langle k,\eta\rangle\frac{|k|^{\tau}}{\gamma}|>\frac{1}{4}$, we have that
\begin{equation}\label{fc5-11}
\|\frac{1-\tilde{l}(\langle k,\eta\rangle\frac{|k|^{\tau}}{\gamma})}{\langle k,\eta\rangle}\|_{D(0,0,0,b),\ell}\leq C\frac{|k|^{(\tau+1)\ell+\tau}}{\gamma^{\ell+1}}.
\end{equation}
The function $g^{(m)}_0=g^{(m)}_0(x,\xi,\eta)$ is analytic in $D(s_m,0,a,b)$, so the Fourier coefficients verify
\begin{equation}\label{fc5-10}
\|\hat{g}_{0}^{(m)}(k,\xi,\eta)\|_{D(0,0,a,b),\ell}\leq \|g_{0}^{(m)}\|_{D(s_m,0,a,b),\ell}e^{-|k|s_m}, \ \forall k\in\mathbf{Z}^{d}.
\end{equation}
Then we can solve homological equation (\ref{fc5-5}) by setting
\begin{equation}\label{fc5-9}
v_{m+1}(x,\xi,\eta)=-\sum_{k\in\mathbf{Z}^{d},\ |\langle k,\eta\rangle\frac{|k|^{\tau}}{\gamma}|>\frac{1}{4} }\frac{(1-\tilde{l}(\langle k,\eta\rangle\frac{|k|^{\tau}}{\gamma}))\hat{g}_{0}^{(m)}(k,\xi,\eta)}{i\langle k,\eta\rangle}e^{i\langle k,x\rangle}.
\end{equation}

By (\ref{fc4-13}), (\ref{fc5-11})-(\ref{fc5-9}), we have
\begin{eqnarray}\label{fc5-12}
 \nonumber&&\|v_{m+1}(x,\xi,\eta)\|_{D(s_m^{(1)},0,a,b),\ell}\\&\leq&\nonumber \sum_{k\in\mathbf{Z}^{d},\ |\langle k,\eta\rangle\frac{|k|^{\tau}}{\gamma}|>\frac{1}{4} }C(\|\hat{g}_{0}^{(m)}(k,\xi,\eta)\|_{D(0,0,a,b),\ell}\|\frac{1-\tilde{l}(\langle k,\eta\rangle\frac{|k|^{\tau}}{\gamma})}{\langle k,\eta\rangle}\|_{D(0,0,0,b),0}\\
&&+\nonumber\|\hat{g}_{0}^{(m)}(k,\xi,\eta)\|_{D(0,0,a,b),0}\|\frac{1-\tilde{l}(\langle k,\eta\rangle\frac{|k|^{\tau}}{\gamma})}{\langle k,\eta\rangle}\|_{D(0,0,0,b),\ell})e^{|k|s_m^{(1)}}\\
&\leq&\nonumber C\sum_{k\in\mathbf{Z}^{d},\ |\langle k,\eta\rangle\frac{|k|^{\tau}}{\gamma}|>\frac{1}{4}}\|g_{0}^{(m)}\|_{D(s_m,0,a,b),\ell}e^{-|k|s_m}\frac{|k|^{(\tau+1)\ell+\tau}}{\gamma^{\ell+1}}e^{|k|s_m^{(1)}}\\
&\leq&\nonumber \frac{C}{\gamma^{\ell+1}}\varepsilon_m\sum_{k\in\mathbf{Z}^{d}/ \left\{0 \right\} }|k|^{(\tau+1)\ell+\tau} e^{-|k|(\frac{s_m-s_{m+1}}{10})}\\
&\leq&\nonumber \frac{C}{\gamma^{\ell+1}}(\frac{1}{s_m-s_{m+1}})^{(\tau+1)\ell+\tau+d}\varepsilon_m\\
&\leq& \frac{C\varepsilon_m}{\gamma^{\ell+1}}.
\end{eqnarray}

By (\ref{fc4-13}), (\ref{fc5-5}), (\ref{fc5-9}) and (\ref{fc5-12}), we have
\begin{equation}\label{fb5-4}
v_{m+1}\in C^{\omega,\infty}(D(s_m^{(1)},0,a,b)),\ v_{m+1}=O^{N+1}(\xi).
\end{equation}
By  (\ref{fc5-12}) and Cauchy estimate, we have
\begin{equation}\label{fb5-1}
\|\partial_{x}v_{m+1}(x,\xi,\eta)\|_{D(s_m^{(2)},0,a,b),\ell}\leq \frac{C\varepsilon_m}{\gamma^{\ell+1}}.
\end{equation}

$\bullet$ Solution to (\ref{fc5-4}).

By (\ref{fc4-21}) and (\ref{fc5-5}), we have
\begin{equation}\label{fc5-13}
 (f_{0}^{(m)}-f_{1}^{(m)}\cdot v_{m+1})^{\hat{}}(0,\xi,\eta)=0.
\end{equation}
By (\ref{fc4-17}), we have
\begin{equation}\label{fc5-14}
 \|f_{1}^{(m)}(x,\xi,\eta)\|_{D(s_m,0,a,b),\ell}\leq C.
\end{equation}
By (\ref{fc4-11}), (\ref{fc5-12}) and (\ref{fc5-14}), we have
\begin{equation}\label{fc5-15}
 \|f_{0}^{(m)}-f_{1}^{(m)}\cdot v_{m+1}\|_{D(s^{(1)}_m,0,a,b),\ell}\leq \frac{C\varepsilon_m}{\gamma^{\ell+1}}.
\end{equation}
Similarly, by applying (\ref{fc5-13}) and (\ref{fc5-15}) to (\ref{fc5-4}), we can solve homological equation (\ref{fc5-4}) by setting
\begin{equation}\label{fb5-100}
u_{m+1}(x,\xi,\eta)=-\sum_{k\in\mathbf{Z}^{d},\ |\langle k,\eta\rangle\frac{|k|^{\tau}}{\gamma}|>\frac{1}{4} }\frac{(1-\tilde{l}(\langle k,\eta\rangle\frac{|k|^{\tau}}{\gamma}))(f_{0}^{(m)}-f_{1}^{(m)}\cdot v_{m+1})^{\hat{}}(k,\xi,\eta)}{i\langle k,\eta\rangle}e^{i\langle k,x\rangle}.
\end{equation}
 Using the same method as in proving (\ref{fc5-12})-(\ref{fb5-1}), we have
\begin{equation}\label{fc5-16}
 \|u_{m+1}\|_{D(s^{(2)}_m,0,a,b),\ell}\leq \frac{C\varepsilon_m}{\gamma^{2\ell+2}},
\end{equation}
\begin{equation}\label{fb5-5}
u_{m+1}\in C^{\omega,\infty}(D(s_m^{(2)},0,a,b)),\ u_{m+1}=O^{N+1}(\xi),
\end{equation}
\begin{equation}\label{fb5-2}
 \|\partial_{x}u_{m+1}\|_{D(s^{(3)}_m,0,a,b),\ell}\leq \frac{C\varepsilon_m}{\gamma^{2\ell+2}}.
\end{equation}

$\bullet$ Solution to (\ref{fc5-6}).

By (\ref{fc4-14}) and (\ref{fc5-9}),we have
\begin{equation}\label{fc5-17}
 v_{m+1}(-x,\xi,\eta)=v_{m+1}(x,\xi,\eta).
\end{equation}
By (\ref{fc4-16}), (\ref{fc4-18}), (\ref{fc4-20}) and (\ref{fc5-17}), we have that
\begin{eqnarray}\label{fc5-18}
 \nonumber&&(g_{1}^{(m)}-2g_{2}^{(m)}\cdot v_{m+1}+\partial_{x}v_{m+1}\cdot f_{1}^{(m)})(-x,\xi,\eta)\\&=& -(g_{1}^{(m)}-2g_{2}^{(m)}\cdot v_{m+1}+\partial_{x}v_{m+1}\cdot f_{1}^{(m)})(x,\xi,\eta),
\end{eqnarray}
which follows
\begin{eqnarray}\label{fc5-21}
(g_{1}^{(m)}-2g_{2}^{(m)}\cdot v_{m+1}+\partial_{x}v_{m+1}\cdot f_{1}^{(m)})^{\hat{}}(0,\xi,\eta)=0.
\end{eqnarray}
By (\ref{fc4-19}) and (\ref{fc5-12}), we have
\begin{equation}\label{fc5-22}
 \|2g_{2}^{(m)}\cdot v_{m+1}\|_{D(s^{(2)}_m,0,a,b),\ell}\leq \frac{C\varepsilon_m}{\gamma^{\ell+1}}.
\end{equation}
By (\ref{fc4-15}), (\ref{fc5-12}), (\ref{fb5-1}), (\ref{fc5-14}) and (\ref{fc5-22}), we have
\begin{equation}\label{fc5-23}
 \|g_{1}^{(m)}-2g_{2}^{(m)}\cdot v_{m+1}+\partial_{x}v_{m+1}\cdot f_{1}^{(m)}\|_{D(s^{(2)}_m,0,a,b),\ell}\leq \frac{C\varepsilon_m}{\gamma^{\ell+1}}.
\end{equation}
Similarly, by applying (\ref{fc5-21}) and (\ref{fc5-23}) to (\ref{fc5-6}), we can solve homological equation (\ref{fc5-6}) by setting
\begin{equation}\label{fb5-101}
 w_{m+1}(x,\xi,\eta)=-\sum_{k\in\mathbf{Z}^{d},\ |\langle k,\eta\rangle\frac{|k|^{\tau}}{\gamma}|>\frac{1}{4} }\frac{(1-\tilde{l}(\langle k,\eta\rangle\frac{|k|^{\tau}}{\gamma}))\hat{G}^{(m)}(k,\xi,\eta)}{i\langle k,\eta\rangle}e^{i\langle k,x\rangle},
\end{equation}
where $G^{(m)}=g_{1}^{(m)}-2g_{2}^{(m)}\cdot v_{m+1}+\partial_{x}v_{m+1}\cdot f_{1}^{(m)}$.
 Using the same method as in proving (\ref{fc5-12})-(\ref{fb5-1}), we have
\begin{equation}\label{fc5-24}
 \|w_{m+1}\|_{D(s^{(3)}_m,0,a,b),\ell}\leq \frac{C\varepsilon_m}{\gamma^{2\ell+2}},
\end{equation}
\begin{equation}\label{fc5-26}
w_{m+1}\in C^{\omega,\infty}(D(s_m^{(3)},0,a,b)),\ w_{m+1}=O^{N+1}(\xi),
\end{equation}
\begin{equation}\label{fc5-25}
 \|\partial_{x}(w_{m+1}\cdot y)\|_{D(s^{(4)}_m,r_m,a,b),\ell}\leq \frac{Cr_m\varepsilon_m}{\gamma^{2\ell+2}}.
\end{equation}

\subsection{Counter-term}
By (\ref{fc5-1})-(\ref{fc5-6}), we have
\begin{eqnarray}\label{fc5-27}
 \nonumber\dot{\theta}&=&\eta+(h_{(m)}+\partial_{x}u_{m+1}\cdot h_{(m)})\cdot (\omega_0-\eta+\Lambda_{(m)}(\xi,\eta))\\
\nonumber&+&P_{\gamma,\tau}(f_{0}^{(m)}-f_{1}^{(m)}\cdot v_{m+1})\\
\nonumber&+&f_{1}^{(m)}\cdot (\rho-w_{m+1} y)+\sum_{|\alpha|\geq2}f_{\alpha}^{(m)}(x,\xi,\eta)y^{\alpha}\\
&+&\partial_{x}u_{m+1}\cdot(f_{0}^{(m)}+\sum_{|\alpha|\geq1}f_{\alpha}^{(m)}(x,\xi,\eta)y^{\alpha})+R_{(m)}^{(1)}+\partial_{x}u_{m+1}\cdot R_{(m)}^{(1)},
\end{eqnarray}

\begin{eqnarray}\label{fc5-28}
 \nonumber\dot{\rho}&=&(H_{(m)}+\partial_{x}v_{m+1}\cdot h_{(m)}+\partial_{x}(w_{m+1} y)\cdot h_{(m)}+w_{m+1}\cdot H_{(m)})\cdot (\omega_0-\eta+\Lambda_{(m)}(\xi,\eta))\\
\nonumber&+&P_{\gamma,\tau}g_{0}^{(m)}+P_{\gamma,\tau}(g_{1}^{(m)}-2g_{2}^{(m)}\cdot v_{m+1}+\partial_{x}v_{m+1}\cdot f_{1}^{(m)})\cdot y\\
\nonumber&+&\sum_{|\beta|=2}g_{\beta}^{(m)}(\rho-w_{m+1} y)^{\beta}-\sum_{|\beta|=2}g_{\beta}^{(m)}v^{\beta}_{m+1}+\sum_{|\beta|\geq3}g_{\beta}^{(m)}(x,\xi,\eta)y^{\beta}\\
\nonumber&+&\partial_{x}v_{m+1}\cdot(f_{0}^{(m)}+\sum_{|\alpha|\geq2}f_{\alpha}^{(m)}(x,\xi,\eta)y^{\alpha})\\
\nonumber&+&\partial_{x}(w_{m+1}\cdot y)\cdot(f_{0}^{(m)}+\sum_{|\alpha|\geq1}f_{\alpha}^{(m)}(x,\xi,\eta)y^{\alpha})\\
\nonumber&+&w_{m+1}\cdot (g_{0}^{(m)}+g_{1}^{(m)}\cdot y+\sum_{|\beta|\geq2}g_{\beta}^{(m)}(x,\xi,\eta)y^{\beta})\\
&+&R_{(m)}^{(2)}+\partial_{x}v_{m+1}\cdot R_{(m)}^{(1)}+\partial_{x}(w_{m+1}\cdot y)\cdot R_{(m)}^{(1)}+w_{m+1}\cdot R_{(m)}^{(2)}.
\end{eqnarray}
By (\ref{fc5-1}), (\ref{fc5-12})-(\ref{fb5-1}), (\ref{fc5-16})-(\ref{fb5-2}), (\ref{fc5-24})-(\ref{fc5-25}) and by means of implicit theorem, we have that $\Phi_{m+1}^{-1}=\Psi_{m+1}$:
 \begin{equation}\label{fc5-29}
\begin{cases}
x=\theta+U_{m+1}(\theta,\xi,\eta),\\
y=\rho+V_{m+1}(\theta,\xi,\eta)+W_{m+1}(\theta,\xi,\eta) \rho,
\end{cases}
\end{equation}
where $(\theta,\rho,\xi,\eta)\in D(s^{(5)}_m,r^{(1)}_m,a,b)$,  $U_{m+1}, V_{m+1}, W_{m+1}\in C^{\omega,\infty}(D(s^{(5)}_m,0,a,b))$,
 $U_{m+1}, V_{m+1}, W_{m+1}=O^{N+1}(\xi)$ and
\begin{equation}\label{fc5-30}
\|U_{m+1}\|_{D(s^{(5)}_m,0,a,b),\ell}\leq \frac{C\varepsilon_m}{\gamma^{2\ell+2}}\ll s_m-s_{m+1},
\end{equation}
\begin{equation}\label{fc5-31}
\|V_{m+1}\|_{D(s^{(5)}_m,0,a,b),\ell}\leq \frac{C\varepsilon_m}{\gamma^{\ell+1}}\ll r_m-r_{m+1},
\end{equation}
\begin{equation}\label{fc5-32}
\|W_{m+1}\|_{D(s^{(5)}_m,0,a,b),\ell}\leq \frac{C\varepsilon_m}{\gamma^{2\ell+2}}\ll r_m-r_{m+1},
\end{equation}
\begin{eqnarray}\label{fc5-33}
 \nonumber\Psi_{m+1}(D(s_{m+1},r_{m+1},a,b))&\subset&\nonumber \Psi_{m+1}(D(s^{(5)}_m,r^{(1)}_m,a,b))\\
&\subset&D(s_{m},r_{m},a,b).
\end{eqnarray}

Inserting (\ref{fc5-29}) into (\ref{fc5-27}) and (\ref{fc5-28}), we can rewrite (\ref{fc5-27}) and  (\ref{fc5-28}) as
\begin{eqnarray}\label{fc5-34}
 \nonumber\dot{\theta}&=&\eta+(h_{(m)}+\partial_{x}u_{m+1}\cdot h_{(m)})(\theta+U_{m+1},\xi,\eta)\cdot (\omega_0-\eta+\Lambda_{(m)}(\xi,\eta))\\
\nonumber&+&(P_{\gamma,\tau}(f_{0}^{(m)}-f_{1}^{(m)}\cdot v_{m+1}))(\theta+U_{m+1},\xi,\eta)\\
\nonumber&+&(f_{1}^{(m)}\cdot (\rho-w_{m+1} (\rho+V_{m+1}+W_{m+1}\rho)))(\theta+U_{m+1},\xi,\eta)\\
\nonumber&+&\sum_{|\alpha|\geq2}f_{\alpha}^{(m)}(\theta+U_{m+1},\xi,\eta)(\rho+V_{m+1}+W_{m+1}\rho)^{\alpha}\\
\nonumber&+&(\partial_{x}u_{m+1}\cdot f_{0}^{(m)})(\theta+U_{m+1},\xi,\eta)\\
\nonumber&+&(\partial_{x}u_{m+1})(\theta+U_{m+1},\xi,\eta)\cdot\sum_{|\alpha|\geq1}f_{\alpha}^{(m)}(\theta+U_{m+1},\xi,\eta)(\rho+V_{m+1}+W_{m+1}\rho)^{\alpha}\\
&+&(R_{(m)}^{(1)}+\partial_{x}u_{m+1}\cdot R_{(m)}^{(1)})(\theta+U_{m+1},\xi,\eta),
\end{eqnarray}

\begin{eqnarray}\label{fc5-35}
 \nonumber\dot{\rho}&=&(H_{(m)}+\partial_{x}v_{m+1}\cdot h_{(m)}+\partial_{x}(w_{m+1}\cdot (\rho+V_{m+1}+W_{m+1}\rho))\cdot h_{(m)}\\
\nonumber&&+w_{m+1}\cdot H_{(m)})(\theta+U_{m+1},\rho+V_{m+1}+W_{m+1}\rho,\xi,\eta)\cdot (\omega_0-\eta+\Lambda_{(m)}(\xi,\eta))\\
\nonumber&+&(P_{\gamma,\tau}g_{0}^{(m)})(\theta+U_{m+1},\xi,\eta)+(P_{\gamma,\tau}(g_{1}^{(m)}-2g_{2}^{(m)}\cdot v_{m+1}\\
\nonumber&&+\partial_{x}v_{m+1}\cdot f_{1}^{(m)}))(\theta+U_{m+1},\xi,\eta)\cdot (\rho+V_{m+1}+W_{m+1}\rho)\\
\nonumber&+&\sum_{|\beta|=2}(g_{\beta}^{(m)}(\rho-w_{m+1} (\rho+V_{m+1}+W_{m+1}\rho))^{\beta}-g_{\beta}^{(m)}v^{\beta}_{m+1})(\theta+U_{m+1},\xi,\eta)\\
\nonumber&+&\sum_{|\beta|\geq3}g_{\beta}^{(m)}(\theta+U_{m+1},\xi,\eta)(\rho+V_{m+1}+W_{m+1}\rho)^{\beta}\\
\nonumber&+&(\partial_{x}v_{m+1}\cdot f_{0}^{(m)})(\theta+U_{m+1},\xi,\eta)\\
\nonumber&+&(\partial_{x}v_{m+1})(\theta+U_{m+1},\xi,\eta)\cdot\sum_{|\alpha|\geq2}f_{\alpha}^{(m)}(\theta+U_{m+1},\xi,\eta)(\rho+V_{m+1}+W_{m+1}\rho)^{\alpha}\\
\nonumber&+&(\partial_{x}(w_{m+1}\cdot (\rho+V_{m+1}+W_{m+1}\rho))\cdot f_{0}^{(m)})(\theta+U_{m+1},\xi,\eta)\\
\nonumber&+&(\partial_{x}(w_{m+1}\cdot (\rho+V_{m+1}+W_{m+1}\rho)))(\theta+U_{m+1},\xi,\eta)\\
\nonumber&&\cdot\sum_{|\alpha|\geq1}f_{\alpha}^{(m)}(\theta+U_{m+1},\xi,\eta)(\rho+V_{m+1}+W_{m+1}\rho)^{\alpha}\\
\nonumber&+&(w_{m+1}\cdot (g_{0}^{(m)}+g_{1}^{(m)}\cdot (\rho+V_{m+1}+W_{m+1}\rho)))(\theta+U_{m+1},\xi,\eta)\\
\nonumber&+&w_{m+1}(\theta+U_{m+1},\xi,\eta)\cdot\sum_{|\beta|\geq2}g_{\beta}^{(m)}(\theta+U_{m+1},\xi,\eta)(\rho+V_{m+1}+W_{m+1}\rho)^{\beta}\\
\nonumber&+&(R_{(m)}^{(2)}+\partial_{x}v_{m+1}\cdot R_{(m)}^{(1)}+\partial_{x}(w_{m+1}\cdot (\rho+V_{m+1}+W_{m+1}\rho))\cdot R_{(m)}^{(1)}\\
&&+w_{m+1}\cdot R_{(m)}^{(2)})(\theta+U_{m+1},\rho+V_{m+1}+W_{m+1}\rho,\xi,\eta).
\end{eqnarray}

Let
\begin{equation}\label{fc5-36}
h_{(j,m+1)}(\theta,\xi,\eta)=h_{(j,m)}(\theta+U_{m+1},\xi,\eta), \ \ j=0,1,2,\cdots,m,
\end{equation}

\begin{equation}\label{fc5-37}
h_{(m+1,m+1)}(\theta,\xi,\eta)=(\partial_{x}u_{m+1}\cdot h_{(m)})(\theta+U_{m+1},\xi,\eta),
\end{equation}

\begin{equation}\label{fc5-38}
H_{(j,m+1)}(\theta,\rho,\xi,\eta)=H_{(j,m)}(\theta+U_{m+1},\rho+V_{m+1}+W_{m+1}\rho,\xi,\eta), \ \ j=0,1,2,\cdots,m,
\end{equation}

\begin{eqnarray}\label{fc5-39}
\nonumber H_{(m+1,m+1)}(\theta,\rho,\xi,\eta)&=&(\partial_{x}v_{m+1}\cdot h_{(m)}+\partial_{x}(w_{m+1} (\rho+V_{m+1}+W_{m+1}\rho))\cdot h_{(m)}\\
&&+w_{m+1}\cdot H_{(m)})(\theta+U_{m+1},\rho+V_{m+1}+W_{m+1}\rho,\xi,\eta),
\end{eqnarray}

\begin{eqnarray}\label{fc5-44}
\nonumber &&\tilde{f}_{0}^{(m+1)}(\theta,\xi,\eta)+\Sigma_{|\alpha|\geq1}f_{\alpha}^{(m+1)}(\theta,\xi,\eta)\rho^{\alpha}\\
\nonumber&=&(f_{1}^{(m)}\cdot (\rho-w_{m+1} (\rho+V_{m+1}+W_{m+1}\rho)))(\theta+U_{m+1},\xi,\eta)\\
\nonumber&+&\sum_{|\alpha|\geq2}f_{\alpha}^{(m)}(\theta+U_{m+1},\xi,\eta)(\rho+V_{m+1}+W_{m+1}\rho)^{\alpha}\\
\nonumber&+&(\partial_{x}u_{m+1}\cdot f_{0}^{(m)})(\theta+U_{m+1},\xi,\eta)\\
&+&(\partial_{x}u_{m+1})(\theta+U_{m+1},\xi,\eta)\cdot\sum_{|\alpha|\geq1}f_{\alpha}^{(m)}(\theta+U_{m+1},\xi,\eta)(\rho+V_{m+1}+W_{m+1}\rho)^{\alpha},
\end{eqnarray}

\begin{eqnarray}\label{fc5-45}
\nonumber &&\tilde{g}_{0}^{(m+1)}(\theta,\xi,\eta)+\tilde{g}_{1}^{(m+1)}(\theta,\xi,\eta) \rho+\sum_{|\beta|\geq2}\tilde{g}_{\beta}^{(m+1)}(\theta,\xi,\eta)\rho^{\beta}\\
\nonumber&=&\sum_{|\beta|=2}(g_{\beta}^{(m)}(\rho-w_{m+1} (\rho+V_{m+1}+W_{m+1}\rho))^{\beta}-g_{\beta}^{(m)}v^{\beta}_{m+1})(\theta+U_{m+1},\xi,\eta)\\
\nonumber&+&\sum_{|\beta|\geq3}g_{\beta}^{(m)}(\theta+U_{m+1},\xi,\eta)(\rho+V_{m+1}+W_{m+1}\rho)^{\beta}\\
\nonumber&+&(\partial_{x}v_{m+1}\cdot f_{0}^{(m)})(\theta+U_{m+1},\xi,\eta)\\
\nonumber&+&(\partial_{x}v_{m+1})(\theta+U_{m+1},\xi,\eta)\cdot\sum_{|\alpha|\geq2}f_{\alpha}^{(m)}(\theta+U_{m+1},\xi,\eta)(\rho+V_{m+1}+W_{m+1}\rho)^{\alpha}\\
\nonumber&+&(\partial_{x}(w_{m+1}\cdot (\rho+V_{m+1}+W_{m+1}\rho))\cdot f_{0}^{(m)})(\theta+U_{m+1},\xi,\eta)\\
\nonumber&+&(\partial_{x}(w_{m+1}\cdot (\rho+V_{m+1}+W_{m+1}\rho)))(\theta+U_{m+1},\xi,\eta)\\
\nonumber&&\cdot\sum_{|\alpha|\geq1}f_{\alpha}^{(m)}(\theta+U_{m+1},\xi,\eta)(\rho+V_{m+1}+W_{m+1}\rho)^{\alpha}\\
\nonumber&+&(w_{m+1}\cdot (g_{0}^{(m)}+g_{1}^{(m)}\cdot (\rho+V_{m+1}+W_{m+1}\rho)))(\theta+U_{m+1},\xi,\eta)\\
&+&w_{m+1}(\theta+U_{m+1},\xi,\eta)\cdot\sum_{|\beta|\geq2}g_{\beta}^{(m)}(\theta+U_{m+1},\xi,\eta)(\rho+V_{m+1}+W_{m+1}\rho)^{\beta},
\end{eqnarray}

\begin{equation}\label{fc5-40}
R^{(1)}_{(j,m+1)}(\theta,\xi,\eta)=R^{(1)}_{(j,m)}(\theta+U_{m+1},\xi,\eta), \ \ j=0,1,2,\cdots,m,
\end{equation}

\begin{eqnarray}\label{fc5-41}
\nonumber R^{(1)}_{(m+1,m+1)}(\theta,\xi,\eta)&=&(\partial_{x}u_{m+1}\cdot R_{(m)}^{(1)})(\theta+U_{m+1},\xi,\eta)\\
&+&(P_{\gamma,\tau}(f_{0}^{(m)}-f_{1}^{(m)}\cdot v_{m+1}))(\theta+U_{m+1},\xi,\eta),
\end{eqnarray}

\begin{equation}\label{fc5-42}
R^{(2)}_{(j,m+1)}(\theta,\rho,\xi,\eta)=R^{(2)}_{(j,m)}(\theta+U_{m+1},\rho+V_{m+1}+W_{m+1}\rho,\xi,\eta), \ \ j=0,1,2,\cdots,m,
\end{equation}

\begin{eqnarray}\label{fc5-43}
\nonumber R^{(2)}_{(m+1,m+1)}(\theta,\rho,\xi,\eta)&=&(\partial_{x}v_{m+1}\cdot R_{(m)}^{(1)}+\partial_{x}(w_{m+1}\cdot (\rho+V_{m+1}+W_{m+1}\rho))\cdot R_{(m)}^{(1)}\\
\nonumber&&+w_{m+1}\cdot R_{(m)}^{(2)})(\theta+U_{m+1},\rho+V_{m+1}+W_{m+1}\rho,\xi,\eta)\\
\nonumber&+&(P_{\gamma,\tau}g_{0}^{(m)})(\theta+U_{m+1},\xi,\eta)+(P_{\gamma,\tau}(g_{1}^{(m)}-2g_{2}^{(m)}\cdot v_{m+1}\\
&&+\partial_{x}v_{m+1}\cdot f_{1}^{(m)}))(\theta+U_{m+1},\xi,\eta)\cdot (\rho+V_{m+1}+W_{m+1}\rho).
\end{eqnarray}

By  Lemma \ref{lem3-2} and (\ref{fc4-5}), (\ref{fc4-7}), (\ref{fc4-8}), (\ref{fc4-10}), (\ref{fc4-11}), (\ref{fc4-13}), (\ref{fc4-15}), (\ref{fc4-17}), (\ref{fc4-19}), (\ref{fc4-22}), (\ref{fc4-24})-(\ref{fb4-10}), (\ref{fb4-12}), (\ref{fb4-13}), (\ref{fc5-12})-(\ref{fb5-1}), (\ref{fc5-16})-(\ref{fb5-2}), (\ref{fc5-24})-(\ref{fc5-25}), (\ref{fc5-29})-(\ref{fc5-32}), (\ref{fc5-36})-(\ref{fc5-43}), we can rewrite (\ref{fc5-34}) and  (\ref{fc5-35}) as
\begin{equation}\label{fc5-44}
\begin{cases}
\dot{\theta}=\eta+h_{(m+1)}\cdot(\omega_0-\eta+\Lambda_{(m)}(\xi,\eta))+\tilde{f}_{0}^{(m+1)}(\theta,\xi,\eta)+\sum_{|\alpha|\geq1}f_{\alpha}^{(m+1)}(\theta,\xi,\eta)\rho^{\alpha}\\ \ \ \ \ \
+R^{(1)}_{(m+1)},\\
\dot{\rho}=H_{(m+1)}\cdot(\omega_0-\eta+\Lambda_{(m)}(\xi,\eta))+\tilde{g}_{0}^{(m+1)}(\theta,\xi,\eta)+\tilde{g}_{1}^{(m+1)}(\theta,\xi,\eta) \rho\\ \ \ \ \ \
 +\sum_{|\beta|\geq2}\tilde{g}_{\beta}^{(m+1)}(\theta,\xi,\eta)\rho^{\beta}+R^{(2)}_{(m+1)},
\end{cases}
\end{equation}
where
\begin{equation}\label{fc5-45}
h_{(m+1)}=\sum_{j=0}^{m+1}h_{(j,m+1)}(\theta,\xi,\eta),\ h_{(0,m+1)}=\mathbf{1}_{d\times d},\ h_{(j,m+1)}\in C^{\omega,\infty}(D(s^{(5)}_m,0,a,b)),
\end{equation}
\begin{equation}\label{fc5-46}
\|h_{(j,m+1)}\|_{D(s^{(5)}_m,0,a,b),\ell}\leq \frac{C\varepsilon_{j-1}}{\gamma^{2\ell+2}},\ j=1,2,\cdots,m+1;
\end{equation}
\begin{equation}\label{fc5-47}
H_{(m+1)}=\sum_{j=0}^{m+1}H_{(j,m+1)}(\theta,\rho,\xi,\eta),\ H_{(0,m+1)}=\mathbf{O}_{d\times d},\
\end{equation}
\begin{equation}\label{fc5-48}
H_{(j,m+1)}\in C^{\omega,\infty}(D(s^{(5)}_m,r^{(1)}_m,a,b)),
\end{equation}
\begin{equation}\label{fc5-49}
 \|H_{(j,m+1)}\|_{D(s^{(5)}_m,r^{(1)}_m,a,b),\ell}\leq \frac{C\varepsilon_{j-1}}{\gamma^{2\ell+2}},\ j=1,2,\cdots,m+1;
\end{equation}
\begin{equation}\label{fc5-50}
\tilde{f}_{0}^{(m+1)}\in C^{\omega,\infty}(D(s^{(5)}_m,0,a,b)),\ \tilde{f}_{0}^{(m+1)}=O^{N+1}(\xi),
\end{equation}
\begin{equation}\label{fb5-30}
 \|\tilde{f}_{0}^{(m+1)}\|_{D(s^{(5)}_m,0,a,b),\ell}\leq \frac{C\varepsilon^2_m}{\gamma^{4\ell+4}}\leq C\varepsilon_{m+1},\
\end{equation}
\begin{equation}\label{fc5-51}
\tilde{g}_{0}^{(m+1)}\in C^{\omega,\infty}(D(s^{(5)}_m,0,a,b)),\ \tilde{g}_{0}^{(m+1)}=O^{N+1}(\xi),
\end{equation}
\begin{equation}\label{fb5-31} \|\tilde{g}_{0}^{(m+1)}\|_{D(s^{(5)}_m,0,a,b),\ell}\leq \frac{C\varepsilon^2_m}{\gamma^{4\ell+4}}\leq C\varepsilon_{m+1},\
\end{equation}
\begin{equation}\label{fc5-52}
\tilde{g}_{1}^{(m+1)}\in C^{\omega,\infty}(D(s^{(5)}_m,0,a,b)),\ \tilde{g}_{1}^{(m+1)}=O^{N+1}(\xi),
\end{equation}
\begin{equation}\label{fb5-32}
 \|\tilde{g}_{1}^{(m+1)}\|_{D(s^{(5)}_m,0,a,b),\ell}\leq \frac{C\varepsilon^2_m}{\gamma^{4\ell+4}}\leq C\varepsilon_{m+1},\
\end{equation}
\begin{equation}\label{fc5-55}
\tilde{g}_{\beta}^{(m+1)}\in C^{\omega,\infty}(D(s^{(5)}_m,0,a,b)),\ |\beta|\geq2,
\end{equation}
\begin{equation}\label{fc5-56}
 \|\sum_{|\beta|\geq2}\tilde{g}_{\beta}^{(m+1)}(\theta,\xi,\eta)\rho^{\beta}\|_{D(s^{(5)}_m,r^{(1)}_m,a,b),\ell}\leq Cr^2_{m+1},
\end{equation}
\begin{equation}\label{fc5-53}
f_{\alpha}^{(m+1)}\in C^{\omega,\infty}(D(s^{(5)}_m,0,a,b)),\ |\alpha|\geq1,
\end{equation}
\begin{equation}\label{fc5-54}
 \|\sum_{|\alpha|\geq1}f_{\alpha}^{(m+1)}(\theta,\xi,\eta)\rho^{\alpha}\|_{D(s^{(5)}_m,r^{(1)}_m,a,b),\ell}\leq Cr_{m+1},
\end{equation}
\begin{equation}\label{fc5-57}
R_{(m+1)}^{(1)}=\sum_{j=0}^{m+1}R_{(j,m+1)}^{(1)}(\theta,\xi,\eta),\ R^{(1)}_{(0,m+1)}=0,
\end{equation}
\begin{equation}\label{fc5-58}
R_{(j,m+1)}^{(1)}\in C^{\omega,\infty}(D(s^{(6)}_m,0,a,b)),
\end{equation}
\begin{equation}\label{fc5-59}
\|R_{(j,m+1)}^{(1)}\|_{D(s^{(6)}_m,0,a,b),\ell}\leq \frac{C\varepsilon_{j-1}}{\gamma^{2\ell+2}},\ j=1,2,\cdots,m+1;
\end{equation}
\begin{equation}\label{fc5-60}
 R_{(m+1)}^{(2)}=\sum_{j=0}^{m+1}R_{(j,m+1)}^{(2)}(\theta,\rho,\xi,\eta),\ R^{(2)}_{(0,m+1)}=0,
\end{equation}
\begin{equation}\label{fc5-61}
R_{(j,m+1)}^{(2)}\in C^{\omega,\infty}(D(s^{(6)}_m,r^{(1)}_m,a,b)),
\end{equation}
\begin{equation}\label{fc5-62}
\|R_{(j,m+1)}^{(2)}\|_{D(s^{(6)}_m,r^{(1)}_m,a,b),\ell}\leq \frac{C\varepsilon_{j-1}}{\gamma^{2\ell+2}},\ j=1,2,\cdots,m+1.
\end{equation}

\begin{lem}\label{lem5-1} $R_{(j,m+1)}^{(1)}$, $R_{(j,m+1)}^{(2)}$ {\rm(}$j=1,2,\cdots,m+1${\rm)} are flat on $DC(\gamma,\tau)$.
\end{lem}
 \begin{proof} By Taylor formula and (\ref{fc5-40}), we have
\begin{eqnarray}\label{fc5-63}
\nonumber R^{(1)}_{(j,m+1)}(\theta,\xi,\eta)&=&R^{(1)}_{(j,m)}(\theta+U_{m+1},\xi,\eta)\\
&=&\sum_{|p|\geq0}\frac{(\partial_{x}^pR^{(1)}_{(j,m)})(\theta,\xi,\eta)}{p!}U_{m+1}^p,
\end{eqnarray}
where $j=1,2,\cdots,m$. By (\ref{fc4-24}), (\ref{fc5-63}) and Definition \ref{defn3-1}, we have that $R_{(j,m+1)}^{(1)}$ ($j=1,2,\cdots,m$) is flat on $DC(\gamma,\tau)$.

By Taylor formula and (\ref{fc5-41}), we have
\begin{eqnarray}\label{fc5-64}
\nonumber R^{(1)}_{(m+1,m+1)}(\theta,\xi,\eta)&=&\partial_{x}u_{m+1}(\theta+U_{m+1},\xi,\eta)\cdot \sum_{|p|\geq0}\frac{(\partial_{x}^pR_{(m)}^{(1)})(\theta,\xi,\eta)}{p!}U_{m+1}^p\\
&+&\sum_{|p|\geq0}\frac{(\partial_{x}^p(P_{\gamma,\tau}(f_{0}^{(m)}-f_{1}^{(m)}\cdot v_{m+1})))(\theta,\xi,\eta)}{p!}U_{m+1}^p.
\end{eqnarray}
By (\ref{fb3-20}), (\ref{fc4-22}), (\ref{fc4-24}), (\ref{fc5-64}) and Definition \ref{defn3-1}, we have that $R_{(m+1,m+1)}^{(1)}$ is flat on $DC(\gamma,\tau)$.

Similarly, by Definition \ref{defn3-1},  (\ref{fb3-20}), (\ref{fc4-22}), (\ref{fc4-24}), (\ref{fb4-10}), (\ref{fb4-12}), (\ref{fc5-42}) and  (\ref{fc5-43}), we have that $R_{(j,m+1)}^{(2)}$ ($j=1,2,\cdots,m+1$) is flat on $DC(\gamma,\tau)$.
  \end{proof}

Now let us consider the reversibility of the changed system. By (\ref{fc4-12}), (\ref{fc4-16}), (\ref{fc4-18}), (\ref{fc4-20}), (\ref{fb5-100}), (\ref{fc5-17}) and (\ref{fb5-101}), we have
\begin{equation}\label{fc5-65}
u_{m+1}(-x,\xi,\eta)=-u_{m+1}(x,\xi,\eta),
\end{equation}
\begin{equation}\label{fc5-66}
w_{m+1}(-x,\xi,\eta)=w_{m+1}(x,\xi,\eta).
\end{equation}
Then by $\Phi_{m+1}\circ\Psi_{m+1}=id$, we get
\begin{equation}\label{fc5-67}
u_{m+1}(x,\xi,\eta)+U_{m+1}(x+u_{m+1}(x,\xi,\eta),\xi,\eta)=0,
\end{equation}
\begin{eqnarray}\label{fc5-68}
\nonumber &&v_{m+1}(x,\xi,\eta)+w_{m+1}(x,\xi,\eta) y+V_{m+1}(x+u_{m+1}(x,\xi,\eta),\xi,\eta)\\
&=&-W_{m+1}(x+u_{m+1}(x,\xi,\eta),\xi,\eta) (y+v_{m+1}(x,\xi,\eta)+w_{m+1}(x,\xi,\eta) y).
\end{eqnarray}
By (\ref{fc5-17}), (\ref{fc5-65})-(\ref{fc5-68}), we have
\begin{equation}\label{fc5-69}
U_{m+1}(-\theta,\xi,\eta)=-U_{m+1}(\theta,\xi,\eta),
\end{equation}
\begin{equation}\label{fc5-70}
V_{m+1}(-\theta,\xi,\eta)=V_{m+1}(\theta,\xi,\eta),
\end{equation}
\begin{equation}\label{fc5-71}
W_{m+1}(-\theta,\xi,\eta)=W_{m+1}(\theta,\xi,\eta),
\end{equation}
where $(\theta,\xi,\eta)\in D(s^{(5)}_m,0,a,b)$. It follows that
\begin{equation}\label{fc5-74}
\tilde{f}_{0}^{(m+1)}(-\theta,\xi,\eta)=\tilde{f}_{0}^{(m+1)}(\theta,\xi,\eta),
\end{equation}
\begin{equation}\label{fc5-75}
\tilde{g}_{0}^{(m+1)}(-\theta,\xi,\eta)=-\tilde{g}_{0}^{(m+1)}(\theta,\xi,\eta),
\end{equation}
\begin{equation}\label{fc5-75-1}
\tilde{g}_{1}^{(m+1)}(-\theta,\xi,\eta)=-\tilde{g}_{1}^{(m+1)}(\theta,\xi,\eta),
\end{equation}
\begin{equation}\label{fc5-75-2}
\tilde{g}_{\beta}^{(m+1)}(-\theta,\xi,\eta)=-\tilde{g}_{\beta}^{(m+1)}(\theta,\xi,\eta), \ \forall|\beta|\geq2,
\end{equation}
\begin{equation}\label{fc5-72}
h_{(j,m+1)}(-\theta,\xi,\eta)=h_{(j,m+1)}(\theta,\xi,\eta),
\end{equation}
\begin{equation}\label{fc5-73}
H_{(j,m+1)}(-\theta,\rho,\xi,\eta)=-H_{(j,m+1)}(\theta,\rho,\xi,\eta),
\end{equation}
\begin{equation}\label{fc5-76}
f_{\alpha}^{(m+1)}(-\theta,\xi,\eta)=f_{\alpha}^{(m+1)}(\theta,\xi,\eta),\ \forall|\alpha|\geq1,
\end{equation}
\begin{equation}\label{fc5-77}
R^{(1)}_{(j,m+1)}(-\theta,\xi,\eta)=R^{(1)}_{(j,m+1)}(\theta,\xi,\eta),
\end{equation}
\begin{equation}\label{fc5-78}
R^{(2)}_{(j,m+1)}(-\theta,\rho,\xi,\eta)=-R^{(2)}_{(j,m+1)}(\theta,\rho,\xi,\eta),
\end{equation}
where $(\theta,\rho,\xi,\eta)\in D(s^{(5)}_m,r^{(1)}_m,a,b)$, $j=1,2,\cdots,m+1$.

\begin{lem}\label{lem5-2} There exists $\Gamma=\Gamma(\xi,\eta)\in C^{\omega,\infty}(D(0,0,a,b))$ with
\begin{equation}\label{fb5-90}
\|\Gamma\|_{D(0,0,a,b),\ell}\leq \frac{C\varepsilon_{m+1}}{\gamma^{\ell+1}},\ \Gamma=O^{N+1}(\xi)
\end{equation}
such that {\rm(\ref{fc5-44})} can be rewritten as
\begin{equation}\label{fc5-79}
\begin{cases}
\dot{\theta}=\eta+h_{(m+1)}\cdot(\omega_0-\eta+\Lambda_{(m+1)}(\xi,\eta))+f_{0}^{(m+1)}(\theta,\xi,\eta)+\sum_{|\alpha|\geq1}f_{\alpha}^{(m+1)}(\theta,\xi,\eta)\rho^{\alpha}\\ \ \ \ \ \
+R^{(1)}_{(m+1)},\\
\dot{\rho}=H_{(m+1)}\cdot(\omega_0-\eta+\Lambda_{(m+1)}(\xi,\eta))+g_{0}^{(m+1)}(\theta,\xi,\eta)+g_{1}^{(m+1)}(\theta,\xi,\eta) \rho\\ \ \ \ \ \
 +\sum_{|\beta|\geq2}g_{\beta}^{(m+1)}(\theta,\xi,\eta)\rho^{\beta}+R^{(2)}_{(m+1)},
\end{cases}
\end{equation}
where $\Lambda_{(m+1)}=\Lambda_{(00)}+\bar{\Lambda}_{(m+1)}$, $\bar{\Lambda}_{(m+1)}=\bar{\Lambda}_{(m)}+\Gamma$ and

\begin{equation}\label{fc5-81}
f_{0}^{(m+1)}\in C^{\omega,\infty}(D(s_{m+1},0,a,b)), f_{0}^{(m+1)}=O^{N+1}(\xi), \|f_{0}^{(m+1)}\|_{D(s_{m+1},0,a,b),\ell}\leq C\varepsilon_{m+1},
\end{equation}
\begin{equation}\label{fc5-82}
f_{0}^{(m+1)}(-\theta,\xi,\eta)=f_{0}^{(m+1)}(\theta,\xi,\eta);
\end{equation}
\begin{equation}\label{fc5-83}
g_{0}^{(m+1)}\in C^{\omega,\infty}(D(s_{m+1},0,a,b)), g_{0}^{(m+1)}=O^{N+1}(\xi), \|g_{0}^{(m+1)}\|_{D(s_{m+1},0,a,b),\ell}\leq C\varepsilon_{m+1},
\end{equation}
\begin{equation}\label{fc5-84}
 g_{0}^{(m+1)}(-\theta,\xi,\eta)=-g_{0}^{(m+1)}(\theta,\xi,\eta);
\end{equation}
\begin{equation}\label{fc5-85}
g_{1}^{(m+1)}\in C^{\omega,\infty}(D(s_{m+1},0,a,b)), g_{1}^{(m+1)}=O^{N+1}(\xi), \|g_{1}^{(m+1)}\|_{D(s_{m+1},0,a,b),\ell}\leq C\varepsilon_{m+1},
\end{equation}
\begin{equation}\label{fc5-86}
 g_{1}^{(m+1)}(-\theta,\xi,\eta)=-g_{1}^{(m+1)}(\theta,\xi,\eta);
\end{equation}
\begin{equation}\label{fc5-87}
g_{\beta}^{(m+1)}\in C^{\omega,\infty}(D(s_{m+1},0,a,b)),\ \forall|\beta|\geq2,
\end{equation}
\begin{equation}\label{fc5-88}
 \|\sum_{|\beta|\geq2}g_{\beta}^{(m+1)}(\theta,\xi,\eta)\rho^{\beta}\|_{D(s_{m+1},r_{m+1},a,b),\ell}\leq Cr^2_{m+1},
\end{equation}
\begin{equation}\label{fc5-89}
 g_{\beta}^{(m+1)}(-\theta,\xi,\eta)=-g_{\beta}^{(m+1)}(\theta,\xi,\eta),\ \forall|\beta|\geq2;
\end{equation}
\begin{equation}\label{fc5-90}
(f_{0}^{(m+1)}-f_{1}^{(m+1)}\cdot \mathcal{L}g_{0}^{(m+1)})^{\hat{}}(0,\xi,\eta)=0.
\end{equation}
\end{lem}
 \begin{proof} Let
 \begin{equation}\label{fc5-91}
\Lambda_{(m+1)}=\Lambda_{(00)}+\bar{\Lambda}_{(m+1)},\ \bar{\Lambda}_{(m+1)}=\bar{\Lambda}_{(m)}+\Gamma,
\end{equation}
 \begin{equation}\label{fc5-92}
h_{(m+1)}\cdot\Gamma(\xi,\eta)+f_{0}^{(m+1)}(\theta,\xi,\eta)=\tilde{f}_{0}^{(m+1)}(\theta,\xi,\eta),
\end{equation}
\begin{eqnarray}\label{fc5-93}
\nonumber &&H_{(m+1)}\cdot\Gamma(\xi,\eta)+g_{0}^{(m+1)}(\theta,\xi,\eta)+g_{1}^{(m+1)}(\theta,\xi,\eta) \rho
+\sum_{|\beta|\geq2}g_{\beta}^{(m+1)}(\theta,\xi,\eta)\rho^{\beta}\\
&&=\tilde{g}_{0}^{(m+1)}(\theta,\xi,\eta)+\tilde{g}_{1}^{(m+1)}(\theta,\xi,\eta) \rho
 +\sum_{|\beta|\geq2}\tilde{g}_{\beta}^{(m+1)}(\theta,\xi,\eta)\rho^{\beta}.
\end{eqnarray}
It follows that (\ref{fc5-44}) can be rewritten as (\ref{fc5-79}). By (\ref{fc5-93}), we have
 \begin{equation}\label{fc5-94}
g_{0}^{(m+1)}(\theta,\xi,\eta)=\tilde{g}_{0}^{(m+1)}(\theta,\xi,\eta)-\bar{H}_{(m+1)}\cdot\Gamma(\xi,\eta),
\end{equation}
where $\bar{H}_{(m+1)}=H_{(m+1)}(\theta,0,\xi,\eta)$.
By (\ref{fc5-47}), (\ref{fc5-75}), (\ref{fc5-73}) and (\ref{fc5-94}), we have
 \begin{equation}\label{fc5-95}
g_{0}^{(m+1)}(-\theta,\xi,\eta)=-g_{0}^{(m+1)}(\theta,\xi,\eta).
\end{equation}
It follows that $\mathcal{L}g_{0}^{(m+1)}$ exists (using the same method as in solving (\ref{fc5-5})). Similarly, by (\ref{fc5-47}), (\ref{fc5-75}) and (\ref{fc5-73}), $\mathcal{L}\tilde{g}_{0}^{(m+1)}$ and $\mathcal{L}\bar{H}_{(m+1)}$ exist. By (\ref{fc5-92}) and (\ref{fc5-94}), we have
\begin{eqnarray}\label{fc5-96}
\nonumber &&(f_{0}^{(m+1)}-f_{1}^{(m+1)}\cdot \mathcal{L}g_{0}^{(m+1)})^{\hat{}}(0,\xi,\eta)\\
\nonumber &=&(\tilde{f}_{0}^{(m+1)}-h_{(m+1)}\cdot\Gamma-f_{1}^{(m+1)}\cdot \mathcal{L}(\tilde{g}_{0}^{(m+1)}-\bar{H}_{(m+1)}\cdot\Gamma))^{\hat{}}(0,\xi,\eta)\\
\nonumber &=&(\tilde{f}_{0}^{(m+1)}-f_{1}^{(m+1)}\cdot \mathcal{L}\tilde{g}_{0}^{(m+1)})^{\hat{}}(0,\xi,\eta)\\
&&-(\mathbf{1}_{d\times d}+(\sum_{j=1}^{m+1}h_{(j,m+1)}-f_{1}^{(m+1)}\cdot \mathcal{L}\bar{H}_{(m+1)})^{\hat{}}(0,\xi,\eta))\cdot\Gamma.
\end{eqnarray}
By (\ref{fc5-47})-(\ref{fc5-49}), (\ref{fc5-51}), (\ref{fb5-31}), (\ref{fc5-75}), (\ref{fc5-73}) and using the same method as in proving (\ref{fc5-12}) and (\ref{fb5-4}), we have
\begin{equation}\label{fc5-97}
\mathcal{L}\tilde{g}_{0}^{(m+1)}\in C^{\omega,\infty}(D(s^{(6)}_m,0,a,b)),\ \mathcal{L}\tilde{g}_{0}^{(m+1)}=O^{N+1}(\xi),
\end{equation}
\begin{equation}\label{fb5-98}
 \|\mathcal{L}\tilde{g}_{0}^{(m+1)}\|_{D(s^{(6)}_m,0,a,b),\ell}\leq \frac{C\varepsilon_{m+1}}{\gamma^{\ell+1}},\
\end{equation}
\begin{equation}\label{fc5-99}
\mathcal{L}\bar{H}_{(m+1)}\in C^{\omega,\infty}(D(s^{(6)}_m,0,a,b)),\  \|\mathcal{L}\bar{H}_{(m+1)}\|_{D(s^{(6)}_m,0,a,b),\ell}\leq \frac{C\varepsilon_0}{\gamma^{3\ell+3}}.
\end{equation}
By (\ref{fc5-45}) and (\ref{fc5-46}), we have
\begin{equation}\label{fc5-100}
\sum_{j=1}^{m+1}h_{(j,m+1)}\in C^{\omega,\infty}(D(s^{(5)}_m,0,a,b)),\ \|\sum_{j=1}^{m+1}h_{(j,m+1)}\|_{D(s^{(5)}_m,0,a,b),\ell}\leq \frac{C\varepsilon_0}{\gamma^{2\ell+2}}.
\end{equation}

Let
\begin{equation}\label{fc5-101}
M(\xi,\eta)=(\tilde{f}_{0}^{(m+1)}-f_{1}^{(m+1)}\cdot \mathcal{L}\tilde{g}_{0}^{(m+1)})^{\hat{}}(0,\xi,\eta),
\end{equation}
\begin{equation}\label{fc5-102}
N(\xi,\eta)=(\sum_{j=1}^{m+1}h_{(j,m+1)}-f_{1}^{(m+1)}\cdot \mathcal{L}\bar{H}_{(m+1)})^{\hat{}}(0,\xi,\eta).
\end{equation}
By (\ref{fc5-50}), (\ref{fb5-30}), (\ref{fc5-53}), (\ref{fc5-54}), (\ref{fc5-97})-(\ref{fc5-102}), we have
\begin{equation}\label{fc5-103}
M\in C^{\omega,\infty}(D(0,0,a,b)),\ M=O^{N+1}(\xi),\ \|M\|_{D(0,0,a,b),\ell}\leq \frac{C\varepsilon_{m+1}}{\gamma^{\ell+1}},
\end{equation}
\begin{equation}\label{fc5-104}
N\in C^{\omega,\infty}(D(0,0,a,b)),\ \|N\|_{D(0,0,a,b),\ell}\leq \frac{C\varepsilon_0}{\gamma^{3\ell+3}}.
\end{equation}

Let
\begin{equation}\label{fc5-105}
M-(\mathbf{1}_{d\times d}+N)\cdot\Gamma=0.
\end{equation}
By (\ref{fc5-103}) and (\ref{fc5-104}), the solution of equation (\ref{fc5-105}) is
\begin{equation}\label{fc5-106}
\Gamma=(\mathbf{1}_{d\times d}+N)^{-1}\cdot M,
\end{equation}
where $\Gamma\in C^{\omega,\infty}(D(0,0,a,b))$, $\|\Gamma\|_{D(0,0,a,b),\ell}\leq \frac{C\varepsilon_{m+1}}{\gamma^{\ell+1}}$, $\Gamma=O^{N+1}(\xi)$.

Then (\ref{fc5-90}) is satisfied by (\ref{fc5-96}), (\ref{fc5-101}), (\ref{fc5-102}), (\ref{fc5-105}) and (\ref{fc5-106}). (\ref{fc5-81}) and (\ref{fc5-82}) are satisfied by (\ref{fc5-45}), (\ref{fc5-46}), (\ref{fc5-50}), (\ref{fb5-30}), (\ref{fc5-74}), (\ref{fc5-72}), (\ref{fc5-92}) and (\ref{fc5-106}). (\ref{fc5-83})-(\ref{fc5-89}) are satisfied by (\ref{fc5-47})-(\ref{fc5-49}), (\ref{fc5-51})-(\ref{fc5-56}), (\ref{fc5-75})-(\ref{fc5-75-2}), (\ref{fc5-73}), (\ref{fc5-93}) and (\ref{fc5-106}). This completes the proof of Lemma \ref{lem5-2}.
  \end{proof}

The proof of iterative lemma (Lemma \ref{lem4-1}) is finished by  (\ref{fc5-45})-(\ref{fc5-49}), (\ref{fc5-53})-(\ref{fc5-62}), (\ref{fc5-72})-(\ref{fc5-78}), Lemma \ref{lem5-1} and Lemma \ref{lem5-2}.

\section{KAM counterterm theorem}\label{sec6}
In the iterative lemma, letting $m\rightarrow\infty$ we get the following KAM counterterm theorem:
\begin{lem}\label{lem6-1}  There exists a d-dimensional vection valued function
\begin{equation}\label{fc6-1}
 \Lambda=\Lambda_{(00)}+\bar{\Lambda}
\end{equation}
with
\begin{equation}\label{fc6-2}
 \Lambda_{(00)}\in C^{\omega,\infty}(D(0,0,a,0)), \ \bar{\Lambda}\in C^{\omega,\infty}(D(0,0,a,b)),
\end{equation}
\begin{equation}\label{fc6-3}
 \Lambda_{(00)}=\sum_{1\leq|\alpha|\leq N}d_{\alpha}\xi^{\alpha},\ \bar{\Lambda}=O^{N+1}(\xi), \  \|\Lambda_{(00)}\|_{D(0,0,a,0),\ell}\leq Ca,\ \|\bar{\Lambda}\|_{D(0,0,a,b),\ell}\leq \frac{C\varepsilon_0}{\gamma^{\ell+1}}
\end{equation}
and there is a diffeomorphism $\Phi=\Psi^{-1}$  with
$$\Psi:D(\frac{s_{0}}{2},\frac{r_{0}}{2},a,b)\rightarrow D(s_{0},r_{0},a,b)$$
of the form
$$\Phi:\theta=\tilde{x}+u(\tilde{x},\xi,\eta), \ \rho=\tilde{y}+v(\tilde{x},\xi,\eta)+w(\tilde{x},\xi,\eta) \tilde{y},$$
where $u, v, w\in C^{\omega,\infty}(D(\frac{s_{0}}{2},0,a,b))$ and
$$\|u\|_{D(\frac{s_{0}}{2},0,a,b),\ell},\ \|v\|_{D(\frac{s_{0}}{2},0,a,b),\ell},\ \|w\|_{D(\frac{s_{0}}{2},0,a,b),\ell}\leq \frac{C\varepsilon_0}{\gamma^{2\ell+2}}$$
such that system {\rm(\ref{fc4-1})} is changed by $\Phi$ into
\begin{equation}\label{fc6-4}
\begin{cases}
\dot{\theta}=\eta+h_{(\infty)}\cdot(\omega_0-\eta+\Lambda(\xi,\eta))+\sum_{|\alpha|\geq1}f_{\alpha}^{(\infty)}(\theta,\xi,\eta)\rho^{\alpha}
+R^{(1)}_{(\infty)},\\
\dot{\rho}=H_{(\infty)}\cdot(\omega_0-\eta+\Lambda(\xi,\eta))
 +\sum_{|\beta|\geq2}g_{\beta}^{(\infty)}(\theta,\xi,\eta)\rho^{\beta}+R^{(2)}_{(\infty)},
\end{cases}
\end{equation}
where
\begin{equation}\label{fc6-5}
h_{(\infty)}\in C^{\omega,\infty}(D(\frac{s_0}{2},0,a,b)),\ \|h_{(\infty)}\|_{D(\frac{s_0}{2},0,a,b),\ell}\leq C,
\end{equation}
 \begin{equation}\label{fb6-1}
h_{(\infty)}(-\theta,\xi,\eta)=h_{(\infty)}(\theta,\xi,\eta),
\end{equation}
\begin{equation}\label{fc6-6}
H_{(\infty)}\in C^{\omega,\infty}(D(\frac{s_0}{2},\frac{r_0}{2},a,b)),\ \|H_{(\infty)}\|_{D(\frac{s_0}{2},\frac{r_0}{2},a,b),\ell}\leq \frac{C\varepsilon_0}{\gamma^{2\ell+2}},
\end{equation}
 \begin{equation}\label{fb6-2}
H_{(\infty)}(-\theta,\rho,\xi,\eta)=-H_{(\infty)}(\theta,\rho,\xi,\eta),
\end{equation}
\begin{equation}\label{fc6-7}
f_{\alpha}^{(\infty)}\in C^{\omega,\infty}(D(\frac{s_0}{2},0,a,b)),\ \forall|\alpha|\geq1,\ \|\sum_{|\alpha|\geq1}f_{\alpha}^{(\infty)}(\theta,\xi,\eta)\rho^{\alpha}\|_{D(\frac{s_0}{2},\frac{r_0}{2},a,b,),\ell}\leq Cr_0,
\end{equation}
 \begin{equation}\label{fb6-3}
f_{\alpha}^{(\infty)}(-\theta,\xi,\eta)=f_{\alpha}^{(\infty)}(\theta,\xi,\eta), \ \forall|\alpha|\geq1,
\end{equation}
\begin{equation}\label{fc6-8}
g_{\beta}^{(\infty)}\in C^{\omega,\infty}(D(\frac{s_0}{2},0,a,b)),\ \forall|\beta|\geq2,\ \|\sum_{|\beta|\geq2}g_{\beta}^{(\infty)}(\theta,\xi,\eta)\rho^{\beta}\|_{D(\frac{s_0}{2},\frac{r_0}{2},a,b),\ell}\leq Cr_0^2,
\end{equation}
 \begin{equation}\label{fb6-4}
g_{\beta}^{(\infty)}(-\theta,\rho,\xi,\eta)=-g_{\beta}^{(\infty)}(\theta,\rho,\xi,\eta), \ \forall|\beta|\geq2,
\end{equation}
\begin{equation}\label{fc6-9}
R_{(\infty)}^{(1)}\in C^{\omega,\infty}(D(\frac{s_0}{2},0,a,b)) \ is \ flat \ on \ DC(\gamma,\tau),\ \|R_{(\infty)}^{(1)}\|_{D(\frac{s_0}{2},0,a,b),\ell}\leq \frac{C\varepsilon_0}{\gamma^{2\ell+2}},
\end{equation}
 \begin{equation}\label{fb6-5}
R_{(\infty)}^{(1)}(-\theta,\xi,\eta)=R_{(\infty)}^{(1)}(\theta,\xi,\eta),
\end{equation}
\begin{equation}\label{fc6-10}
R_{(\infty)}^{(2)}\in C^{\omega,\infty}(D(\frac{s_0}{2},\frac{r_0}{2},a,b)) \ is \ flat \ on \ DC(\gamma,\tau),\ \|R_{(\infty)}^{(2)}\|_{D(\frac{s_0}{2},\frac{r_0}{2},a,b),\ell}\leq \frac{C\varepsilon_0}{\gamma^{2\ell+2}},
\end{equation}
 \begin{equation}\label{fb6-6}
R_{(\infty)}^{(2)}(-\theta,\rho,\xi,\eta)=-R_{(\infty)}^{(2)}(\theta,\rho,\xi,\eta).
\end{equation}
\end{lem}

\section{Proof of  Theorem 1.1}\label{sec7}
In this section, we will prove Theorem 1.1. First, we give some lemmas.
\begin{lem}\label{lem7-1} Assume that $\omega_0\in DC(\gamma_0,\tau_0)$. There exist
\begin{equation}\label{fc7-1}
\Psi:x=\theta+u(\theta,\rho,\xi),\
y=\rho+\xi+v(\theta,\rho,\xi),
\end{equation}
where $u(-\theta,\rho,\xi)=-u(\theta,\rho,\xi)$, $v(-\theta,\rho,\xi)=v(\theta,\rho,\xi)$, $\hat{v}(0,\rho,\xi)=0$, $u=\sum_{|\alpha|+|\bar{\alpha}|\geq1}u_{\alpha,\bar{\alpha}}(\theta)\rho^{\alpha}\xi^{\bar{\alpha}}$,
$v=\sum_{|\beta|+|\bar{\beta}|\geq2}v_{\beta,\bar{\beta}}(\theta)\rho^{\beta}\xi^{\bar{\beta}}$, $u_{\alpha,\bar{\alpha}}, v_{\beta,\bar{\beta}}\in C^{\omega,\infty}(D(s,0,0,0))$ {\rm(}$|\alpha|+|\bar{\alpha}|\geq1$, $|\beta|+|\bar{\beta}|\geq2${\rm)} and $\Omega(\xi)=\Sigma_{|\alpha|\geq0}\tilde{d}_{\alpha}\xi^{\alpha}$ ($\tilde{d}_{\alpha}$ are constants) such that system {\rm(\ref{fc1-3})} is changed by $\Phi:=\Psi^{-1}$ into
 \begin{equation}\label{fc7-2}
\begin{cases}
\dot{\theta}=\Omega(\xi)+F(\theta,\rho,\xi)\rho,\\
\dot{\rho}=\langle G(\theta,\rho,\xi)\rho, \rho\rangle:=G(\theta,\rho,\xi)\rho^2,
\end{cases}
\end{equation}
where $F(\theta,\rho,\xi)=\Sigma_{|\alpha|+|\bar{\alpha}|\geq0}F_{\alpha,\bar{\alpha}}(\theta)\rho^{\alpha}\xi^{\bar{\alpha}}$, $G(\theta,\rho,\xi)=\Sigma_{|\beta|+|\bar{\beta}|\geq0}G_{\beta,\bar{\beta}}(\theta)\rho^{\beta}\xi^{\bar{\beta}}$, $F_{\alpha,\bar{\alpha}}, G_{\beta,\bar{\beta}}\in C^{\omega,\infty}(D(s,0,0,0))$, $F_{\alpha,\bar{\alpha}}(-\theta)=F_{\alpha,\bar{\alpha}}(\theta)$, $G_{\beta,\bar{\beta}}(-\theta)=-G_{\beta,\bar{\beta}}(\theta)$. Then, $\Omega(\xi)=\Sigma_{|\alpha|\geq0}\tilde{d}_{\alpha}\xi^{\alpha}$ is unique.
\end{lem}
\begin{proof}
By  (\ref{fc7-1}), we have that
\begin{equation}\label{fc7-3}
\begin{cases}
\dot{x}=\dot{\theta}+\partial_{\theta}u\cdot\dot{\theta}+\partial_{\rho}u\cdot\dot{\rho},\\
\dot{y}=\dot{\rho}+\partial_{\theta}v\cdot\dot{\theta}+\partial_{\rho}v\cdot\dot{\rho}.
\end{cases}
\end{equation}
By (\ref{fc1-3}), (\ref{fc7-1})-(\ref{fc7-3}), we have that
\begin{eqnarray}\label{fc7-4}
\nonumber &&\Omega(\xi)+F(\theta,\rho,\xi)\rho+\partial_{\theta}u\cdot(\Omega(\xi)+F(\theta,\rho,\xi)\rho)+\partial_{\rho}u\cdot(G(\theta,\rho,\xi)\rho^2)\\
&=&\omega_0+f(\theta+u(\theta,\rho,\xi),\rho+\xi+v(\theta,\rho,\xi)),
\end{eqnarray}
\begin{eqnarray}\label{fc7-5}
\nonumber &&G(\theta,\rho,\xi)\rho^2+\partial_{\theta}v\cdot(\Omega(\xi)+F(\theta,\rho,\xi)\rho)+\partial_{\rho}v\cdot(G(\theta,\rho,\xi)\rho^2)\\ &=&g(\theta+u(\theta,\rho,\xi),\rho+\xi+v(\theta,\rho,\xi)).
\end{eqnarray}
Let
$$F_{[j]}(\theta,\rho,\xi)=[F(\theta,\rho,\xi)]_j$$
be the homogeneous component of degree $j$ of $F(\theta,\rho,\xi)$. Similarly, we can define $G(\theta,\rho,\xi)$, $u(\theta,\rho,\xi)$ and so on.

Let
$$\mathcal{M}(F(\theta,\rho,\xi))=\hat{F}(0,\rho,\xi)$$
be the mean value of $F(\theta,\rho,\xi)$, Similarly, we can define $G(\theta,\rho,\xi)$, $u(\theta,\rho,\xi)$ and so on.

For $j=0$, the equation (\ref{fc7-4}) becomes
\begin{equation}\label{fc7-6}
\Omega_{[0]}(\xi)=\omega_0.
\end{equation}

For $j=1$, the equation (\ref{fc7-4}) becomes
\begin{equation}\label{fc7-7}
\Omega_{[1]}(\xi)+F_{[0]}(\theta,\rho,\xi)\rho+\omega_0\cdot\partial_{\theta}u_{[1]}=f_{[1]}(\theta,\rho+\xi),
\end{equation}
where $\omega_0=(\omega_{01}, \omega_{02},\cdots,\omega_{0d})^{T}$, $\omega_0\cdot \partial_{\theta}=\sum_{j=1}^d\omega_{0j}\partial_{\theta_j}$.
Then we must have
\begin{equation}\label{fc7-8}
\Omega_{[1]}(\xi)=\mathcal{M}(f_{[1]}(\theta,\xi)),
\end{equation}
which  determines $\Omega_{[1]}(\xi)$ uniquely.
Then we get the equation for $u_{[1]}$:
\begin{equation}\label{fc7-9}
\omega_0\cdot\partial_{\theta}u_{[1]}=f_{[1]}(\theta,\xi+\rho)-\mathcal{M}(f_{[1]}(\theta,\xi))-F_{[0]}(\theta,\rho,\xi)\rho.
\end{equation}
 We have $\hat{u}(0,\rho,\xi)=0$ by $u(-\theta,\rho,\xi)=-u(\theta,\rho,\xi)$, then $u_{[1]}$ is unique modulo $O^1(\rho)$. (But we can add any term of degree one in $O^1(\rho)$ to $u_{[1]}$ by changing the definition of $F_{[0]}$.)

 For $j=2$, the equations (\ref{fc7-4}) and (\ref{fc7-5}) become
\begin{eqnarray}\label{fc7-10}
\nonumber &&
\Omega_{[2]}(\xi)+F_{[1]}\cdot\rho+\omega_0\cdot\partial_{\theta}u_{[2]}+\partial_{\theta}u_{[1]}\cdot\Omega_{[1]}(\xi)+\partial_{\theta}u_{[1]}\cdot (F_{[0]}\cdot\rho)+\partial_{\rho}u_{[1]}\cdot[G\rho^2]_0\\
&&=[(f_{[1]}+f_{[2]})(\theta+u_{[1]},\rho+\xi+v_{[2]})]_2,
\end{eqnarray}
\begin{equation}\label{fc7-11}
[G(\theta,\rho,\xi)\rho^2]_2+\omega_0\cdot\partial_{\theta}v_{[2]}=g_{[2]}(\theta,\rho+\xi).
\end{equation}

Noting $\hat{v}(0,\rho,\xi)=0$, the solution $v_{[2]}$ of (\ref{fc7-11}) is unique modulo $O^1(\rho)$.

By (\ref{fc7-10}), we must have
\begin{equation}\label{fc7-12}
\Omega_{[2]}(\xi)+\mathcal{M}(\partial_{\theta}u_{[1]}(\theta,0,\xi)\cdot\Omega_{[1]}(\xi))=\mathcal{M}([(f_{[1]}+f_{[2]})(\theta+u_{[1]}(\theta,0,\xi),\xi+v_{[2]}(\theta,0,\xi))]_2),
\end{equation}
which  determines $\Omega_{[2]}(\xi)$ uniquely. Then we get the equation for $u_{[2]}$:
\begin{eqnarray}\label{fc7-13}
\nonumber &&
F_{[1]}\cdot\rho+\omega_0\cdot\partial_{\theta}u_{[2]}+\partial_{\theta}u_{[1]}\cdot\Omega_{[1]}(\xi)-\mathcal{M}(\partial_{\theta}u_{[1]}(\theta,0,\xi)\cdot\Omega_{[1]}(\xi))+\partial_{\theta}u_{[1]}\cdot (F_{[0]}\cdot\rho)\\
\nonumber &=&-\partial_{\rho}u_{[1]}\cdot[G\rho^2]_2+[(f_{[1]}+f_{[2]})(\theta+u_{[1]},\rho+\xi+v_{[2]})]_2\\
&&-\mathcal{M}([(f_{[1]}+f_{[2]})(\theta+u_{[1]}(\theta,0,\xi),\xi+v_{[2]}(\theta,0,\xi))]_2).
\end{eqnarray}
Clearly, $u_{[2]}$ is unique modulo $O^1(\rho)$.

We now proceed by induction on $j\geq3$: assume that for $j\leq m$, $\Omega_{[j]}(\xi)$ is unique, $u_{[j]}$ and $v_{[j]}$ are unique modulo $O^1(\rho)$. We have seen that this induction assumption is true for $m=2$.

 For $j=m+1$, the equations (\ref{fc7-4}) and (\ref{fc7-5}) become
\begin{equation}\label{fc7-14}
\Omega_{[m+1]}(\xi)+F_{[m]}\cdot\rho+\omega_0\cdot\partial_{\theta}u_{[m+1]}+P_{[m+1]}(\theta,\rho,\xi)+Q_{[m+1]}(\theta,\rho,\xi)
=R_{[m+1]}(\theta,\rho,\xi),
\end{equation}
where $$P_{[m+1]}=\partial_{\theta}u_{[1]}\cdot\Omega_{[m]}(\xi)+\cdots+\partial_{\theta}u_{[m]}\cdot\Omega_{[1]}(\xi),$$
$$Q_{[m+1]}=\partial_{\theta}u_{[1]}\cdot (F_{[m-1]}\rho)+\cdots+\partial_{\theta}u_{[m]}\cdot (F_{[0]}\rho)+\partial_{\rho}u_{[1]}\cdot[G\rho^2]_{m+1}+\cdots+\partial_{\rho}u_{[m]}\cdot[G\rho^2]_2,$$
$$R_{[m+1]}=[(f_{[1]}+\cdots+f_{[m+1]})(\theta+u_{[1]}+\cdots+u_{[m]},\rho+\xi+v_{[2]}+\cdots+v_{[m+1]})]_{m+1},$$
\begin{equation}\label{fc7-15}
[G(\theta,\rho,\xi)\rho^2]_{m+1}+\omega_0\cdot\partial_{\theta}v_{[m+1]}+S_{[m+1]}(\theta,\rho,\xi)+\tilde{S}_{[m+1]}(\theta,\rho,\xi)
=T_{[m+1]}(\theta,\rho,\xi),
\end{equation}
where
$$S_{[m+1]}=\partial_{\theta}v_{[2]}\cdot\Omega_{[m-1]}(\xi)+\cdots+\partial_{\theta}v_{[m]}\cdot\Omega_{[1]}(\xi),$$
$$\tilde{S}_{[m+1]}=\partial_{\theta}v_{[2]}\cdot (F_{[m-2]}\rho)+\cdots+\partial_{\theta}v_{[m]}\cdot (F_{[0]}\rho)+\partial_{\rho}v_{[2]}\cdot[G\rho^2]_m+\cdots+\partial_{\rho}v_{[m]}\cdot[G\rho^2]_2,$$
$$T_{[m+1]}=[(g_{[2]}+\cdots+g_{[m+1]})(\theta+u_{[1]}+\cdots+u_{[m]},\rho+\xi+v_{[2]}+\cdots+v_{[m]})]_{m+1}.$$
Noting $\hat{v}(0,\rho,\xi)=0$, the solution $v_{[m+1]}$ of (\ref{fc7-15}) is unique modulo $O^1(\rho)$. By (\ref{fc7-14}), we must have
\begin{equation}\label{fc7-16}
\Omega_{[m+1]}(\xi)
=\mathcal{M}(R_{[m+1]}(\theta,0,\xi)-P_{[m+1]}(\theta,0,\xi)),
\end{equation}
which  determines $\Omega_{[m+1]}(\xi)$ uniquely. Then we get the equation for $u_{[m+1]}$:
\begin{eqnarray}\label{fc7-17}
\nonumber &&
\mathcal{M}(R_{[m+1]}(\theta,0,\xi)-P_{[m+1]}(\theta,0,\xi))+F_{[m]}\cdot\rho+\omega_0\cdot\partial_{\theta}u_{[m+1]}\\
 &=&R_{[m+1]}(\theta,\rho,\xi)-P_{[m+1]}(\theta,\rho,\xi)-Q_{[m+1]}(\theta,\rho,\xi).
\end{eqnarray}
Clearly, $u_{[m+1]}$ is unique modulo $O^1(\rho)$.

This shows the existence of $\Omega$ verifying (\ref{fc7-2}) up to any order $j$, as well
as the uniqueness.
\end{proof}

Noting Lemma \ref{lem2-3}, let $\theta=\phi$, $\rho=\mu-\xi$. By (\ref{fb2-100}), we have that
 \begin{equation}\label{fb7-1}
\begin{cases}
\dot{\theta}=f_F(\xi)+\tilde{F}(\rho,\xi)\rho,\\
\dot{\rho}=0,
\end{cases}
\end{equation}
where $\tilde{F}(\rho,\xi)=\Sigma_{|\alpha|+|\bar{\alpha}|\geq0}\tilde{F}_{\alpha,\bar{\alpha}}\rho^{\alpha}\xi^{\bar{\alpha}}$, $\tilde{F}_{\alpha,\bar{\alpha}}$ ($|\alpha|+|\bar{\alpha}|\geq0$) are constants. It follows that, for any $j\geq1$,
\begin{equation}\label{fb7-2}
\Omega(\xi)=f_F(\xi)+O^j(\xi).
\end{equation}
\begin{lem}\label{lem7-2}  There exists
a unique $C^{\infty}$ function $\eta(\xi)$: $\mathbf{R}^d\rightarrow \mathbf{R}^d$, defined in the neighborhood of 0 given by $|\xi|<a$ such that
\begin{equation}\label{fc7-18}
\omega_0-\eta(\xi)+\Lambda(\xi,\eta(\xi))=0,
\end{equation}
where $\eta(\xi)=\omega_0+\sum_{1\leq|\alpha|\leq N}d_{\alpha}\xi^{\alpha}+G(\xi)$, $G(\xi)=O^{N+1}(\xi)$, $\|G(\xi)\|_{D(0,0,a,b),\ell}\leq \frac{C\varepsilon_0}{\gamma^{\ell+1}}$.
Moreover, if $\eta(\xi)\in DC(\gamma,\tau)$, the Taylor series of $\eta(\xi)$ at $\xi=0$ is given by $f_F(\xi)$.
\end{lem}
\begin{proof}By (\ref{fc6-1})-(\ref{fc6-3}) and the implicit function theorem, we can get (\ref{fc7-18}). If $\eta(\xi)\in DC(\gamma,\tau)$,  we can rewrite (\ref{fc6-4}) as
\begin{equation}\label{fc7-19}
\begin{cases}
\dot{\theta}=\eta(\xi)+\sum_{|\alpha|\geq1}f_{\alpha}^{(\infty)}(\theta,\xi,\eta(\xi))\rho^{\alpha},\\
\dot{\rho}=\sum_{|\beta|\geq2}g_{\beta}^{(\infty)}(\theta,\xi,\eta(\xi))\rho^{\beta}.
\end{cases}
\end{equation}
By Lemma \ref{lem2-3}, Remark \ref{rem2-4}, Lemma \ref{lem7-1}, (\ref{fb7-1})-(\ref{fc7-19}), the Taylor series of $\eta(\xi)$ at $\xi=0$ is given by $f_F(\xi)$.
\end{proof}

By Lemma \ref{lem7-2}, if $\eta(\xi)\in DC(\gamma,\tau)$, $\left\{\eta(\xi)t:t\in\mathbf{R} \right\}\times\left\{\rho=0 \right\}$ is an analytic invariant torus of (\ref{fc6-4}).

\begin{lem}\label{lem7-3}{\rm (Lemma 5.1 in \cite{a2})}  If {\rm(\ref{fb2-100})} is 0-degenerate, then there exist $p,\sigma>0$ such that for any  $k\in\mathbf{Z}^{d}/ \left\{0 \right\}$
there exists a unit vector $u_k\in\mathbf{R}^{d}$ such that the series
$$f_k(\mu)=\langle\frac{k}{|k|},f_F(\mu)\rangle$$
is $(p,\sigma)$-transverse in direction $u_k$; that is,
$$\max_{0\leq j\leq p}|\partial_t^jf_k(tu_k)_{|t=0}|\geq\sigma.$$
\end{lem}

Assume that
\begin{equation}\label{fff7-1}
\|\eta(\xi)-\Sigma_{0\leq j\leq p}[f_{F}(\xi)]_j\|_{C^p(D(0,0,a,0))}\leq\frac{\sigma}{2}.
\end{equation}

\begin{lem}\label{lem7-4}{\rm (Lemma 5.2 in \cite{a2})} If $f_F(\mu)$ is $(p,\sigma)$-transverse (in some direction), then
$$Leb \left\{|\xi|<a: |\langle\frac{k}{|k|},\eta(\xi)\rangle|<\varepsilon\right\}\leq C(\frac{\varepsilon}{\sigma})^{\frac{1}{p}}a^{d-1}$$
for any $a,k,\varepsilon$.
\end{lem}

If (\ref{fb2-100}) is 0-degenerate, we can assume without restriction that $\sigma\leq1$, let $\tau=dp+1$, $a=(\frac{\gamma}{\sigma})^{\frac{1}{2p}}$. Let $N>\max\left\{60(\ell+1)p,10d\right\}$. Then by (\ref{fc7-18}) we have that
$$\|\eta(\xi)-\Sigma_{0\leq j\leq p}[f_{F}(\xi)]_j\|_{C^p(D(0,0,a,0))}\leq Ca\leq\frac{\sigma}{2}$$
if $a$ is sufficiently small.
Then by    Lemma \ref{lem7-4}, we have that
$$Leb \left\{|\xi|<a: \eta(\xi)\not\in DC(\gamma,\tau)\right\}\leq C(\frac{\gamma}{\sigma})^{\frac{1}{p}}a^{d-1}\leq CaLeb \left\{|\xi|<a\right\},$$
provided $\gamma$ is sufficiently small. Hence, the set
$$\left\{|\xi|<a: \eta(\xi)\in DC(\gamma,\tau)\right\}$$
is of positive measure and density 1 at 0 as $\gamma\rightarrow0$. This completes the proof of conclusion (i) of Theorem \ref{thm1-1}.

Then we will prove conclusions (ii) and (iii) of Theorem \ref{thm1-1}. The proof is the same as Section 6 in \cite{a2}.  Let $\tau=\tau_0$, $\gamma=\frac{\gamma_0}{2}$. By Lemma \ref{lem7-2}, if $\eta(\xi)\in DC(\gamma,\tau)$, we have that
\begin{equation}\label{fb7-200}
\omega_0-\eta(\xi)+\Lambda(\xi,\eta(\xi))=0,
\end{equation}
\begin{equation}\label{fb7-201}
\eta(\xi)=f_F(\xi)+O^{\infty}(\xi).
\end{equation}
By $\omega_0=(\omega_{01},\omega_{02},\cdots,\omega_{0d})^{T}$, $\eta=(\eta_{1},\eta_{2},\cdots,\eta_{d})^{T}$, $\Lambda=(\Lambda_{1},\Lambda_{2},\cdots,\Lambda_{d})^{T}$, $f_F=((f_F)_{1},(f_F)_{2},\cdots,(f_F)_{d})^{T}$, (\ref{fb7-200}) and (\ref{fb7-201}) can be written as
\begin{equation}\label{fd7-1}
\begin{cases}
\omega_{01}-\eta_1(\xi)+\Lambda_1(\xi,\eta(\xi))=0,\\
\omega_{02}-\eta_2(\xi)+\Lambda_2(\xi,\eta(\xi))=0,\\
\cdots\\
\omega_{0d}-\eta_d(\xi)+\Lambda_d(\xi,\eta(\xi))=0,
\end{cases}
\end{equation}

\begin{equation}\label{fd7-2}
\begin{cases}
\eta_1(\xi)=(f_F)_1(\xi)+O^{\infty}(\xi),\\
\eta_2(\xi)=(f_F)_2(\xi)+O^{\infty}(\xi),\\
\cdots\\
\eta_d(\xi)=(f_F)_d(\xi)+O^{\infty}(\xi).
\end{cases}
\end{equation}
When (\ref{fb2-100}) is $j$-degenerate ($1\leq j\leq d-1$), let $\partial_v$ be the directional derivative in direction $v$, then we have
\begin{equation}\label{fb7-202}
\partial_v^n\langle f_{F}(\xi), v\rangle=\partial_v^n(\sum_{i=1}^d(f_{F})_i(\xi)\cdot v_i)=0, \ \ \ \ \forall n\geq0,
\end{equation}
for every $\xi=(\xi_1,\xi_2,\cdots,\xi_d)^{T}\backsim0\in \mathbf{R}^{d}$ and any $v=(v_1,v_2,\cdots,v_d)^{T}\in {\rm Lin}(\tilde{\gamma}=(\gamma_1,\cdots,\gamma_j))$.

For $\forall1\leq i\leq d$, by (\ref{fb1-100}) and (\ref{fb7-202}), we have that
\begin{equation}\label{fd7-3}
\partial_v(f_F)_i(\xi)=\sum_{l=1}^d\frac{\partial (f_{F})_i(\xi)}{\partial \xi_{l}}\cdot v_{l}=\sum_{l=1}^d\frac{\partial (f_{F})_l(\xi)}{\partial \xi_{i}}\cdot v_{l}=\frac{\partial(\sum_{l=1}^d (f_{F})_l(\xi)\cdot v_{l})}{\partial \xi_{i}}=0.
\end{equation}
Then we have
\begin{equation}\label{fd7-4}
\partial_v^n(f_F)_i(\xi)=0, \ \ \ \ \forall n\geq1.
\end{equation}
By (\ref{fd7-2}) and (\ref{fd7-4}), we have
\begin{equation}\label{fd7-5}
\partial_v^n\eta_i(\xi)_{|\xi=0}=0, \ \ \ \ \forall n\geq1.
\end{equation}
By (\ref{fd7-1}), we have
\begin{equation}\label{fd7-6}
\partial_v(\omega_{0i}-\eta_i(\xi)+\Lambda_i(\xi,\eta(\xi)))=0
\end{equation}
and
\begin{eqnarray}\label{fd7-7}
\nonumber &&\partial_v(\omega_{0i}-\eta_i(\xi)+\Lambda_i(\xi,\eta(\xi)))\\
\nonumber&=&-\partial_v\eta_i(\xi)+\sum_{l=1}^d\frac{\partial \Lambda_i}{\partial \xi_{l}}\cdot v_{l}+\sum_{l=1}^d(\sum_{p=1}^d\frac{\partial \Lambda_i}{\partial \eta_{p}}\cdot\frac{\partial \eta_p}{\partial \xi_{l}})\cdot v_{l}\\
\nonumber&=&-\partial_v\eta_i(\xi)+\sum_{l=1}^d\frac{\partial \Lambda_i}{\partial \xi_{l}}\cdot v_{l}+\sum_{p=1}^d\frac{\partial \Lambda_i}{\partial \eta_{p}}(\sum_{l=1}^d\frac{\partial \eta_p}{\partial \xi_{l}}\cdot v_{l})\\
&=&-\partial_v\eta_i(\xi)+\sum_{l=1}^d\frac{\partial \Lambda_i}{\partial \xi_{l}}\cdot v_{l}+\sum_{p=1}^d\frac{\partial \Lambda_i}{\partial \eta_{p}}\cdot\partial_v \eta_p.
\end{eqnarray}
By (\ref{fd7-5}), we have
\begin{equation}\label{fd7-8}
(-\partial_v\eta_i(\xi)+\sum_{p=1}^d\frac{\partial \Lambda_i}{\partial \eta_{p}}\cdot\partial_v \eta_p)_{|\xi=0}=0.
\end{equation}
By (\ref{fc6-1}), (\ref{fc6-3}) and (\ref{fb7-200}), we have
\begin{equation}\label{fd7-9}
\eta(0)=\omega_0
\end{equation}
and
\begin{equation}\label{fd7-10}
\Lambda_i(0,\omega_0)=0.
\end{equation}
Then by (\ref{fd7-6})-(\ref{fd7-9}), we have
\begin{equation}\label{fb7-203}
\partial_v(\Lambda_i(\cdot,\omega_0))_{|\xi=0}=0.
\end{equation}
Similarly, we have
\begin{equation}\label{fd7-11}
\partial_v^n(\Lambda_i(\cdot,\omega_0))_{|\xi=0}=0,\ \ \ \ \forall n\geq2.
\end{equation}
 By (\ref{fd7-10})-(\ref{fd7-11}), since $s\mapsto\Lambda(\langle s,\tilde{\gamma}\rangle,\omega_0)$ is an analytic function in $s\in \mathbf{R}^{j}$ , $s\backsim0$, it must be identically 0, and hence $\eta(\langle s,\tilde{\gamma}\rangle)$ is identically $\omega_0$; that is,
$$\eta(\langle s,\tilde{\gamma}\rangle)\in DC(\gamma,\tau)$$
for all sufficiently small $s$. It follows that for any $\xi=\langle s,\tilde{\gamma}\rangle\in {\rm Lin}(\tilde{\gamma}=(\gamma_1,\cdots,\gamma_j))$ with sufficiently small $s$, (\ref{fc1-3}) has an
analytic invariant torus with frequency $\omega_0$ and that the set of all these tori is a $(d+j)$-dimensional subvariety.
This completes the proof of conclusion (ii) of Theorem \ref{thm1-1}.
\begin{rem}\label{rem7-5} By the above process of the proof, (\ref{fb1-100}) is a sufficient but unnecessary condition for the conclusion (ii) of Theorem \ref{thm1-1} (we get (\ref{fd7-3}) and (\ref{fd7-4}) by (\ref{fb1-100})). There are many reversible systems which are $j$-degenerate ($1\leq j\leq d-1$) and do not satisfy condition (\ref{fb1-100}). For example, consider the following analytic $G$-reversible system of the form (\ref{fc1-3}):
$$\dot{x}_1=f_1(y)=\omega_{01}(1+y_1^2+y_2^2+\cdots+y_d^2),$$
$$\cdots$$
$$\dot{x}_{i}=f_{i}(y)=\omega_{0i}(1+y_1^2+y_2^2+\cdots+y_d^2),\ \ \ i=j+1,$$
$$\dot{x}_{i}=f_{i}(y)=\omega_{0i}+y_{j+2},\ \ \ \ i=j+2,$$
$$\cdots$$
$$\dot{x}_{d}=f_{d}(y)=\omega_{0d}+y_{d},$$
$$\dot{y}=0.$$
This system coincides with its BNF around $\mathbf{T}^{d}\times\left\{y=0 \right\}$ and is $j$-degenerate because the identity
$\Sigma_{i=1}^d\gamma_if_i(y)\equiv0$ ($\gamma\in \mathbf{R}^{d}$) is equivalent to that
$$\sum_{i=1}^{j+1}\gamma_i\omega_{0i}=0,\ \ \ \ \gamma_{j+2}=\cdots=\gamma_d=0.$$
The condition (\ref{fb1-100}) is not satisfied, for instance,
$$\frac{\partial f_1(y)}{\partial y_2}=2\omega_{01}y_2,\ \ \ \frac{\partial f_2(y)}{\partial y_1}=2\omega_{02}y_1.$$
Each torus $\mathbf{T}^{d}\times\left\{y=const \right\}$ is an invariant torus. Hence, for an analytic $G$-reversible system of the form (\ref{fc1-3}) which is $j$-degenerate ($1\leq j\leq d-1$) and does not satisfy condition (\ref{fb1-100}), maybe through $\Gamma_0$ of the reversible system there passes an analytic subvariety of dimension $d+j$ foliated into analytic invariant tori, even, a full neighborhood of $\Gamma_0$ may be foliated into analytic invariant tori. Unfortunately, for an analytic $G$-reversible system of the form (\ref{fc1-3}) which is $j$-degenerate ($1\leq j\leq d-1$), I can not prove  a universal conclusion without condition (\ref{fb1-100}).
\end{rem}

When (\ref{fb2-100}) is $d-1$-degenerate, then
$$f_F(\xi)=\mu(\langle\xi,\omega_0\rangle)\omega_0,$$
where $\mu(t)=1+O^1(t)$ is a formal power series in one variable.

Since
$$\omega_0-\mu(\langle\xi,\omega_0\rangle)\omega_0+\Lambda(\xi,\mu(\langle\xi,\omega_0\rangle)\omega_0)=O^{\infty}(\xi),$$
taking $\xi=t\omega_0$, we have (assuming that $\omega_0$ is a unit vector)
\begin{equation}\label{fc7-20}
\omega_0-\mu(t)\omega_0+\Lambda(t\omega_0,\mu(t)\omega_0)=0
\end{equation}
modulo a term in $O^{\infty}(t)$. Since, by (\ref{fc6-1})-(\ref{fc6-3}), the left-hand side is analytic
in $t\omega_0$ and $\mu$, we obtain from any of the equations (\ref{fc7-20}) that $\mu(t)$ is a convergent power series. Then
$$t\mapsto\omega_0-\mu(t)\omega_0+\Lambda(t\omega_0,\mu(t)\omega_0)$$
is analytic for $t\sim0$, and hence identically zero. We derive from this that
$$\eta(\xi)=\mu(\langle\xi,\omega_0\rangle)\omega_0;$$
that is,
$$\eta(\xi)\in DC(\gamma,\tau)$$
for all sufficiently small $\xi$. This completes the proof of conclusion (iii) of Theorem \ref{thm1-1}.
\bibliographystyle{abbrv} 
\bibliography{kam}

\vspace{6pt}



\begin{thebibliography}{aa}
\bibitem[1]{a2}
L. H. Eliasson, B. Fayad, R. Krikorian, Around the stability of KAM tori, {\it
Duke Mathematical Journal}, {\bf 164(9)}(2015), 1733-1775.
\bibitem[2]{b1}
H. R\"ussmann, \"Uber die Normalform analytischer Hamiltonscher Differentialgleichungen
in der N\"ahe einer Gleichgewichtsl\"osung, {\it Math. Ann.}, {\bf 169} (1967), 55-72.
\bibitem[3]{a1}
M. Herman, Some open problems in dynamical systems, Proceedings of the International Congress of Mathematicians, Vol. II, Berlin, 1998, Doc. Math. 1998 Extra Vol. II, 797-808.
\bibitem[4]{b2}
A. Bounemoura, Non-degenerate Liouville tori are KAM stable, {\it Adv. Math.}, {\bf  292} (2016), 42-51.
\bibitem[5]{b3}
J. Moser, Combination tones for Duffing's equation, {\it Comm. Pure Appl. Math.}, {\bf  18} (1965), 167-181.
\bibitem[6]{a19}
J. Moser, Stable and Random Motion in Dynamic Systems, Ann. of Math. Studies, Princeton Uni. Press,
Princeton, NJ, 1973.
\bibitem[7]{a4}
V. I. Arnold, Reversible systems,
Nonlinear and turbulent processes in physics, Acad. Publ., New York, 1984, pp1161-1174.

\bibitem[8]{a5}
V. I. Arnold, M. B. Sevryuk, Oscillations and bifurcations in reversible systems,
 Nonlinear phenomena in plasma physics and hydrodynamics, Mir Publishers, Moscow, 1986.

\bibitem[9]{a6}
H. W. Broer,  G. B. Huitema,  M. B. Sevryuk, Quasi-periodic motions in families of dynamical systems: order amidst chaos. Springer, Berlin, 1996.
\bibitem[10]{a11}
M. B. Sevryuk, Reversible Systems, Lecture Notes in Math., vol. 1211, Springer-Verlag, New York/Berlin,
1986.
\bibitem[11]{a12}
M. B. Sevryuk, The iteration-approximation decoupling in the reversible KAM theory, {\it Chaos}, {\bf 5(3)} (1995), 552-565.

\bibitem[12]{a13}
M. B. Sevryuk, The finite-dimensional reversible KAM theory, {\it Physica D}, {\bf 112(1-2)} (1998), 132-147.


\bibitem[13]{a14}
M. B. Sevryuk, The reversible context 2 in KAM theory: the first steps, {\it Regular and Chaotic Dynamics}, {\bf 16(1-2)} (2011), 24-38.

\bibitem[14]{a17}
J. Zhang,  On lower dimensional invariant tori in $C^d$ reversible systems.{\it  Chinese Annals of Mathematics,} Series B, {\bf 29(5)} (2008), 459-486.


\bibitem[15]{a18}
J. Li, J. Qi, X. Yuan, KAM theorem for reversible mapping of low smoothness
with application, arXiv:submit/2890714 [math.DS] 18 Oct 2019.

\bibitem[16]{a7}
B. Liu, An application of KAM theorem of reversible systems, {\it Sci. China Ser. A.}, {\bf 34(9)} (1991), 1068-1078.

\bibitem[17]{a8}
B. Liu, F. Zanolin, Boundedness of solutions of nonlinear differential equations,
{\it J. Diff. Eqs.} {\bf 144} (1998), 66-98.

\bibitem[18]{a9}
B. Liu, Invariant curves of quasi-periodic reversible mappings, {\it Nonlinearity} {\bf 18} (2005), 685-701.

    \bibitem[19]{a10}
D. Piao, W. Li, Boundedness of solutions for reversible system via Moser's twist theorem,
{\it J. Math. Anal. Appl.}, {\bf 341(2)} (2008), 1224-1235.

\bibitem[20]{a16}
R. Yuan, X. Yuan, Boundedness of solutions for a class of nonlinear differential equations
 of second order via Moser's twist theorem, {\it Nonlinear
Anal.} {\bf 46(8)} (2001), 1073-1087.


\end{thebibliography}

\end{document}